\documentclass[10pt]{article}
\usepackage{amsfonts,amssymb,amsmath,amsthm,epsfig}
\usepackage{youngtab}

\theoremstyle{plain}
\newtheorem{thm}{Theorem}[section]

\newtheorem{prop}[thm]{Proposition}
\newtheorem{lem}[thm]{Lemma}
\theoremstyle{definition}

\newtheorem{example}[thm]{\sc Example} 
\newtheorem{rem}[thm]{\sc Remark}

\newcommand{\Comp}{\mbox{${\mathbb C}$}}
\newcommand{\bfalpha}{\mbox{\boldmath $\alpha$}}
\newcommand{\bfbeta}{\mbox{\boldmath $\beta$}}

\newcommand{\bfLambda}{\mbox{\boldmath $\Lambda$}}

\newcommand{\ssupset}{\mbox{$\vartriangleright$}}
\newcommand{\ssubset}{\mbox{$\vartriangleleft$}}

\newcommand{\abox}{\mbox{\tiny\yng(1)}}
\newcommand{\horibox}{\mbox{\tiny\yng(2)}}
\newcommand{\vertbox}{\mbox{\tiny\yng(1,1)}}

\newcommand{\fig}[3]{
	\begin{figure}[htbp]
		\begin{center}
			\includegraphics{#3}	
			\caption{#2}
			\label{#1}
		\end{center}
	\end{figure}
}

\begin{document}

\title{\bf Partition Algebra,\\
Its Characterization and Representations
}

\author{Masashi KOSUDA}

\date{}
\maketitle

\begin{center}
{\bf Abstract}
\end{center}

In this note we give representations for
the partition algebra $A_{3}(Q)$ in Young's seminormal form.
For this purpose, we also give characterizations
of $A_{n}(Q)$ and $A_{n-\frac{1}{2}}(Q)$.



\section{Introduction}
\subsection{Definition of the partition algebra}
Let $M = \{1, 2, \ldots, n\}$ be
a set of $n$ symbols and $F = \{1', \ldots, n'\}$ another set
of $n$ symbols.
We assume that the elements of $M$ and $F$
are ordered by $1<2<\cdots <n$ and $1'<2'<\cdots<n'$ respectively. 
Consider the following set of set partitions:
\begin{eqnarray}
\Sigma_n^1
&=&
\{\{T_1, \ldots, T_s\}\ |\ s=1,2, \ldots \ ,\nonumber\\
& &	\quad T_j(\neq\emptyset)\subset M\cup F\ 
	(j = 1, 2, \ldots, s),\\
& &	\quad \cup T_j = M\cup F,\quad T_i\cap T_j = \emptyset \mbox{ if }
	i\neq j\nonumber\}.
\end{eqnarray}
We call an element $w$ of $\Sigma_n^1$ {\em a seat-plan}
and each element of $w$ a {\em part} of $w$.
It is easy to see that the number of seat-plans is equal
to $B_{2n}$, the Bell number.

For $w\in\Sigma_n^1$
consider a rectangle with $n$ marked points on the bottom and the same
$n$ on the top as in Figure~\ref{fig:plan}.
\fig{fig:plan}{A seat-plan of $\Sigma_5$}{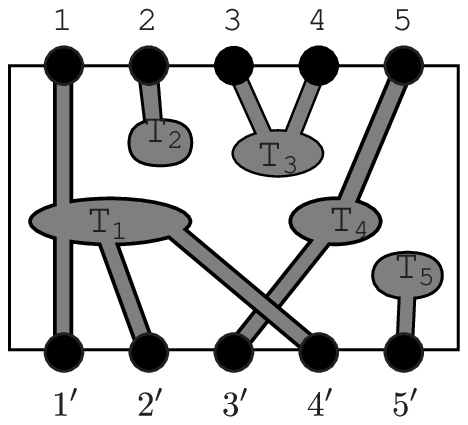}
The $n$ marked points on the top are labeled by
$1, 2, \ldots n$
from left to right.
Similarly, the $n$ marked points on the bottom are labeled by
$1', 2', \ldots, n'$.
If $w$ consists of $s$ parts,
then put $s$ shaded circles
in the middle of the rectangle
so that they have no intersections.
Then we join the $2n$ marked points and the $s$ circles
with $2n$ shaded bands so that
each shaded circle represent a part of $w$.

Using these diagrams, for $w_1,w_2\in\Sigma_n^1$,
an arbitrary pair of seat-plans,
we can define a product $w_1 w_2$.
The product is obtained by placing $w_1$ on $w_2$,
gluing the corresponding boundaries
and shrinking half along the vertical axis.
We then have a new diagram possibly containing some
shaded regions which are not connected to the boundaries.
If the resulting diagram has $p$ such regions,
then the product is defined by
the diagram with such region removed
and multiplied by $Q^p$.
Here $Q$ is an indeterminate.
(It is easily checked that the product defined above is closed
in the linear span of the set of seat-plans $\Sigma_n^1$
over $\mathbb{Z}[Q]$.)
For example, if
\[
	w_1 = \{\{1,1',4'\}, \{2,5\},\{3,4\}, \{2'\},\{3',5'\}\}\in\Sigma_5^1
\]
and
\[
	w_2 = \{\{1,1',3',4'\}, \{2\}, \{3,5\}, \{4\}, \{2',5'\}\}\in\Sigma_5^1,
\]
then we have
\[
	w_1w_2 =
	Q^2\{\{ 1,1',3',4'\}, \{2,5\}, \{3,4\}, \{2',5'\}\}
	\in \mathbb{Z}[Q]\Sigma_5^1
\]
as in Figure~\ref{fig:prod}.
\fig{fig:prod}{The product of seat-plans}{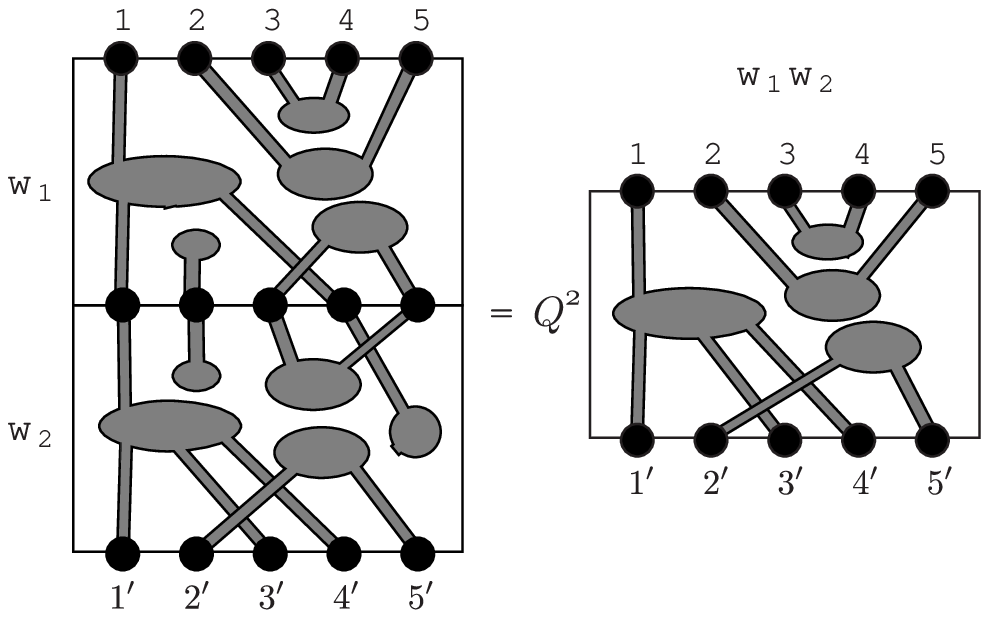}
By this product,
the set of linear combinations of the elements of $\Sigma_n^1$
over $\mathbb{Z}[Q]$ makes an algebra $A_{n}(Q)$
called the {\em partition algebra}.
The identity of $A_{n}(Q)$
is a diagram which corresponds to the partition
\[
1 = \{\{1, 1'\}, \{2, 2'\}, \ldots, \{n, n'\}\}.
\]
We put $A_{0}(Q) = A_{1}(Q)= \mathbb{Z}[Q]$.
We can define $A_{n}(Q)$ more rigorously in terms of the set partitions
(See P.~P.~Maritin's paper~\cite{Ma2}).

Next we define special elements $s_i, f_i$ ($1\leq i \leq n-1$)
and $e_i$ ($1\leq i\leq n$) of $\Sigma_n^1$
by
\begin{eqnarray*}
s_i &=& \{\{1,1'\},\ldots, \{i-1,(i-1)'\},
		\{i+2, (i+2)'\},\ldots, \{n, n'\},\\
	& &\quad \{i, (i+1)'\}, \{i+1, i'\}\}\\
f_i &=& \{\{1,1'\},\ldots, \{i-1,(i-1)'\},
		\{i+2, (i+2)'\},\ldots, \{n, n'\},\\
	& &\quad \{i, i+1, i', (i+1)'\}\}\\
e_i &=& \{\{1,1'\},\ldots, \{i-1,(i-1)'\}, \{i\}, \{i'\}
		\{i+1, (i+1)'\},\ldots, \{n, n'\}\}.
\end{eqnarray*}
The diagrams of these special elements are
illustrated by the figures in Figure~\ref{fig:gen}.
Note that in the picture of $e_i$,
there exist ``a male'' only part and
``a female'' only part.
We call such a part ``defective'' (see Section~3.1).
\fig{fig:gen}{Special elements}{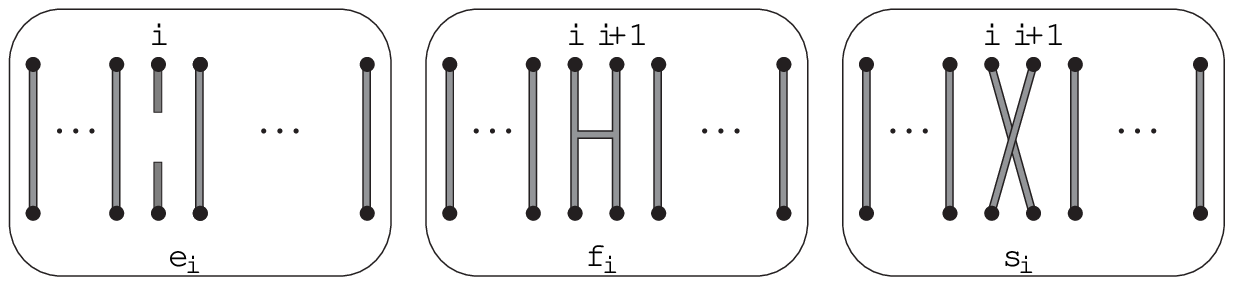}

We easily find that they satisfy the following basic relations.
\begin{equation}
	\begin{array}{rcl}\tag{$R0$}
	f_{i+1} &=& s_is_{i+1}f_is_{i+1}s_i \quad (i = 1, 2, \ldots, n-2),\\
	e_{i+1} &=& s_ie_{i}s_i \quad (i = 1, 2, \cdots, n-1)
	\end{array}
\end{equation}
\begin{equation}
	\begin{array}{rcl}\tag{$R1$}
	s_i^2 &=& 1 \quad (i = 1, 2, \ldots, n-1),\\
	s_is_{i+1}s_i &=& s_{i+1}s_is_{i+1} \quad (i = 1, 2, \ldots, n-2),\\
	s_is_j &=& s_js_i \quad (|i-j|\geq 2),
	\end{array}
\end{equation}
\begin{equation}
	f_i^2 = f_i,\ f_if_j = f_jf_i,\tag{$R2$}
\end{equation}
\begin{equation}
	f_is_i = s_if_i=f_i,\tag{$R3$}
\end{equation}
\begin{equation}
	f_is_j = s_jf_i \quad (|i -j|\geq 2),\tag{$R4$}
\end{equation}
\begin{equation}
	e_{i}^2 = Qe_i ,\tag{$E1$}
\end{equation}
\begin{equation}
	s_ie_{i}e_{i+1}
	= e_ie_{i+1}s_i
	= e_ie_{i+1}
		\quad (i = 1, 2, \ldots, n-1),\tag{$E2$}
\end{equation}
\begin{equation}
	e_is_j = s_je_i
		\quad(j-i\geq1,\ i-j\geq 2),
	\quad e_ie_j
		= e_je_i,\tag{$E3$}
\end{equation}
\begin{equation}
	\begin{array}{rcl}\tag{$E4$}
	e_if_ie_i = e_i 
	& e_{i+1}f_{i}e_{i+1} = e_{i+1}
	& (i = 1, 2, \ldots, n-1),\\
	f_ie_{i}f_i = f_{i},
	&f_{i}e_{i+1}f_{i} = f_{i}
	&(i = 1, 2, \ldots, n-1),
	\end{array}
\end{equation}
\begin{equation}
	e_if_j = f_je_i
		\quad(j-i\geq1,\ i-j\geq 2).\tag{$E5$}
\end{equation}
Here we make a remark on the special elements above.
\begin{rem}\label{rem:gen}
The relation $(R0)$ implies that the special elements
$\{f_i\}$
and $\{e_i\}$
are generated by $f=f_1$, $e = e_1$ and $s_1,\ldots, s_{n-1}$.
\end{rem}

In this note, firstly we show that
the special elements and the basic relations ($R0$)-($R4$) 
and ($E1$)-($E5$) above characterize
the partition algebra $A_{n}(Q)$, {\em i.e.}
the special elements generate $A_{n}(Q)$,
and all the possible relations in $A_{n}(Q)$ are
obtained from the basic relations.
By Remark~\ref{rem:gen},
the basic relations will be translated into the relations
among the symbols $f$, $e$ and $s_i$s.
Characterizations will be stated by these symbols.

\subsection{Characterization for $A_{n}(Q)$}
Since generators $\{s_i\ | \ 1\leq i\leq n-1\}$
of the partition algebra $A_{n}(Q)$
satisfy the relations of the symmetric group ${\mathfrak S}_n$,
we can understand that $f_i$ and $e_i$ are ``conjugate'' to $f$ and $e$
respectively.

Hence the basic relations $(R2)$-$(R4)$ and $(E1)$-$(E5)$
among the special elements
are translated into the relations $(R2')$-$(R4')$
and $(E1')$-$(E5')$ among the generators
as follows.
\begin{thm}\label{th:main}
The partition algebra $A_{n}(Q)$ is
characterized by the generators
\begin{equation*}
	f, e, s_1, s_2, \ldots, s_{n-1},
\end{equation*}
and the relations
\begin{equation}
	\begin{array}{rcl}\tag{$R1$}
	s_i^2 &=& 1 \quad (i = 1, 2, \ldots, n-1),\\
	s_is_{i+1}s_i &=& s_{i+1}s_is_{i+1} \quad (i = 1, 2, \ldots, n-2),\\
	s_is_j &=& s_js_i \quad (|i-j|\geq 2, \ i, j =1,2, \ldots, n-1),
	\end{array}
\end{equation}
\begin{equation}
	f^2 = f,\ fs_2fs_2 = s_2fs_2f,\ fs_2s_1s_3s_2fs_2s_1s_3s_2 = s_2s_1s_3s_2fs_2s_1s_3s_2f,\tag{$R2'$}
\end{equation}
\begin{equation}
	fs_1 = s_1f=f,\tag{$R3'$}
\end{equation}
\begin{equation}
	fs_i = s_if \quad (i = 3, 4, \ldots, n-1),\tag{$R4'$}
\end{equation}
\begin{equation}
	e^2 = Qe,\tag{$E1'$}
\end{equation}
\begin{equation}
	es_1es_1 = s_1es_1e = es_1e,\tag{$E2'$}
\end{equation}
\begin{equation}
	es_i=s_ie \quad (i = 2, 3, \ldots, n-1),\tag{$E3'$}
\end{equation}
\begin{equation}
	efe = e,\ fef = f,\tag{$E4'$}
\end{equation}
\begin{equation}
	fs_2s_1es_1s_2
	= s_2s_1es_1s_2f.\tag{$E5'$}\end{equation}
\end{thm}

In Sections~2-4
we prove this theorem not using
the generators and the relations in the theorem
but using the special elements and the basic relations $(R0)$-$(R4)$
and $(E1)$-$(E5)$.

The partition algebras $A_n(Q)$ were introduced in early 1990s
by Martin~\cite{Ma1,Ma2} and Jones~\cite{Jo} independently
and have been studied, for example, in the papers~\cite{Ma3,DW,HR}. 
The theorem above has already shown in the paper~\cite{HR}.
Here we give another poof defining a ``standard'' expression of
a word of the special elements of $A_{n}(Q)$
according to the papers~\cite{Ko1,Ko3,Ko4}.
From this standard expression,
we will find that the partition algebra $A_{n}(Q)$
is cellular in the sense of Graham and Lehrer~\cite{GL}.
Thus, applying the general representation of cellular
algebras to the partition algebras,
we will get a description of the irreducible
modules of $A_{n}(Q)$ for any field of arbitrary characteristic.
(For the cell representations, we also refer the paper~\cite{KL}.)

Further, we can make the character table of $A_{n}(Q)$
using the standard expressions.
These topics will be studied in near future.
For the present we refer the notes~\cite{Ko3,Na1}
and the results about the partition algebras~\cite{DW,Xi}.

\section{Local moves deduced from the basic relations}

Let
\[
{\cal L}_n^1 =
\{
	s_1, s_2, \ldots, s_{n-1},
	f_1, f_2, \ldots, f_{n-1},
	e_1, e_2, \ldots, e_{n}
\}
\]
be the set of symbols
whose words satisfy the basic relations $(R0)$-$(R4)$
and $(E1)$-$(E5)$.
There are many relations among these symbols
which are deduced from the basic relations.
These relations are pictorially
expressed as local moves.
Among them, we frequently use relations
$f_{i+1}s_is_{i+1}=s_is_{i+1}f_i$ ($R0$), 
$f_is_{i+1}f_i = f_if_{i+1}$ ($R2''$)
and
$e_is_i = e_if_ie_{i+1} =s_ie_{i+1}$ ($E4''$)
as in Figure~\ref{fig:fss},\ref{fig:fsf} and \ref{fig:efe}
respectively.
The latter two relations are deduced from the relations ($R0$)-($R3$)
and ($R0$), ($R3$), ($E4$) respectively.
\fig{fig:fss}{$f_{i+1}s_is_{i+1}=s_is_{i+1}f_i$ ($R0$)}{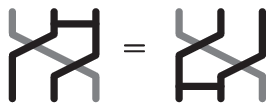}
\fig{fig:fsf}{$f_is_{i+1}f_i = f_if_{i+1}$ ($R2''$)}{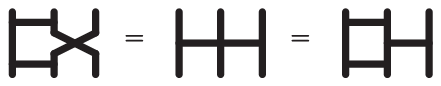}
\fig{fig:efe}{$e_is_{i}=e_if_ie_{i+1} = s_ie_{i+1}$ ($E4''$)}{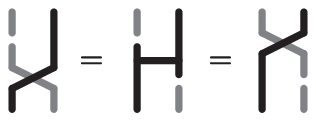}

As we noted in the previous paper~\cite{Ko4},
these basic relations
are invariant under the transpositions of indices
$i\leftrightarrow n-i+1$ as well as
the $\mathbb{Z}[Q]$-linear involution $*$ defined by
$(xy)^{*} = y^*x^*$ ($x, y\in A_{n}(Q)$).
This implies that if one local move is allowed
then other three moves ---obtained by
reflections with respect to the vertical and the horizontal lines
and their composition---
are also allowed.

Further, we note that if we put
\[
	e^{[r]} = f_1f_2\cdots f_{r-1}e_1e_2\cdots e_rf_1f_2\cdots f_{r-1}
\]
then we can check that $e^{[r]}$, $f$ and $s_i$ ($1\leq i\leq n-1$)
satisfy the defining relations of $P_{n,r}(Q)$,
the $r$-modular party algebra, defined in the paper~\cite{Ko4}.
This means that the local moves shown in the paper~\cite{Ko4}
also hold in $A_n(Q)$
(in fact, these local moves are more easily verified
in $A_n(Q)$).
Some of them are pictorially expressed
in Figures~\ref{fig:dpjrE},\ref{fig:roeE},\ref{fig:dpsE}
and \ref{fig:dpeE}.
\fig{fig:dpjrE}{Defective part jump rope\ ($R13'$)}{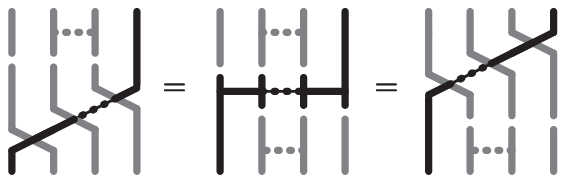}
\fig{fig:roeE}{Removal (addition) of excrescences\ ($R14'$)}{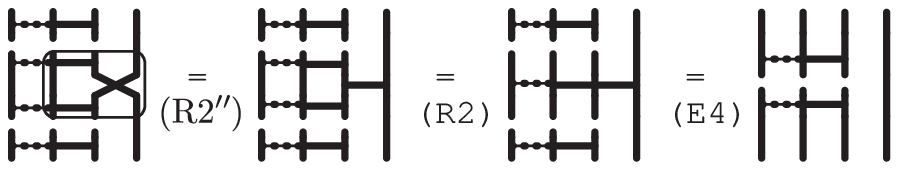}
\fig{fig:dpsE}{Defective part shift\ ($R16'$)}{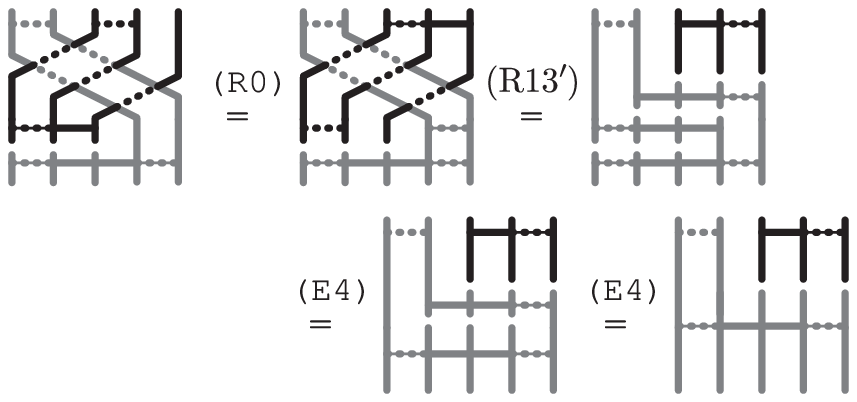}
\fig{fig:dpeE}{Defective part exchange\ ($R17'$)}{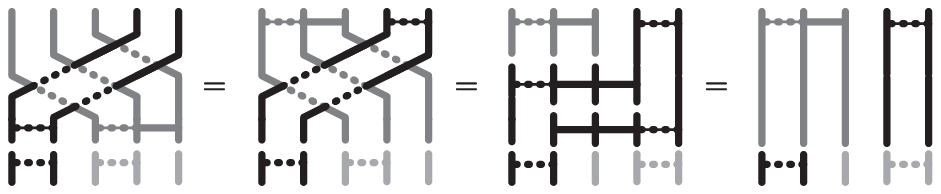}

\section{Standard expressions of seat-plans}

In this section, for a seat-plan $w$ of $\Sigma_n^1$,
we define a {\em basic expression},
as a word in the alphabet $\Gamma_n^1$.
Then we define more general forms
called {\em  crank form expression}s.
As a special type of the crank form expression,
we define the {\em standard expression}.
In the next section,
we show that any
two crank form expressions of a seat-plan
will be moved to each other by using the basic relations
$(R0)$-$(R4)$ and $(E1)$-$(E5)$
finite times.
Consequently, we find that
any seat-plan can be moved to its standard expression.
To define these expressions,
we introduce some terminologies.

\subsection{Propagating number}

Let $w = \{T_1, T_2, \ldots, T_s\}$
be a seat-plan of $A_{n}(Q)$.
For a part $T_i\in w$, the intersection with $M$, or $T_i^M = M\cap T_i$,
is called the {\em upper part} of $T_i$.
Similarly, $T_i^F = F\cap T_i$
is called the {\em lower part} of $T_i$.
If $T_i^M\neq\emptyset$ and $T_i^F\neq\emptyset$ hold
simultaneously,
$T_i$ is called {\em propagating},
otherwise, it is called {\em non-propagating},
or {\em defective}.
Let
$\pi(w) := \{T\in w\ |\ T:\mbox{propagating}\}$
be the set of propagating parts.
The number of propagating parts $|\pi(w)|$ is
called the {\em propagating number} (of $w$).
If $T_i\in\pi(w)$ then the upper [resp. lower]
part $T^M_i$ [resp. $T^F_i$] of $T_i$ is also called {\em propagating}.
If $T_i\in w\setminus\pi(w)$ and $T^M_i = T_i$ [resp. $T^F_i = T_i$],
then $T^M_i$ [resp. $T^F_i$] is called {\em defective}.

For example, in Figure~\ref{fig:plan},
$\pi(w) = \{T_1, T_4\}$.
Hence $|\pi(w)| = 2$.
On the other hand $T_2$, $T_3$ and $T_5$ are defective.
The upper and the lower propagating parts are $\{1\}$, $\{5\}$
and $\{1',2',4'\}$, $\{3'\}$ respectively.
The upper defective parts are $T_2$ and $T_3$.
The lower defective part is $T_5$.

\subsection{A basic expression of a seat-plan}

For a part $T_i\in w$, define $t(T_i)$ by
\[
	t(T_i) =
	\left\{
	\begin{array}{ll}
		1&\mbox{if $T_i$ is propagating,}\\
		0&\mbox{if $T_i$ is defective.}
	\end{array}
	\right.
\]
Similarly we define $t(T_i^M)$ [resp. $t(T_i^F)$]
to be $1$ or $0$ in accordance with that $T_i^M$ [resp. $T_i^F$] is propagating
or not.


Using the terminologies above,
first we define a {\em basic expression} of an seat-plan.
Let $w\in\Sigma_n^1$ be a seat-plan
and $\rho_w = (T_1, \ldots, T_s)$ be an arbitrary sequence of all parts of $w$.
For the sequence $\rho_w$,
we define the sequence of the upper [resp. lower]
parts $\mathbb{M} = \mathbb{M}(\rho_w) = (T^M_{i_1}, \ldots, T^M_{i_u})$
 ($i_1<\cdots<i_u$, $u\leq s$)
[resp. $\mathbb{F} = \mathbb{F}(\rho_w) = (T^F_{j_1}, \ldots, T^F_{j_v})$
 ($j_1<\cdots<j_v$, $v\leq s$)] omitting empty parts.

Using these data, we define {\em cranks}
$C_{\mathbb{M}}[i]$, $C^*_{\mathbb{F}}[i]$
and $C^{\mathbb{M}}_{\mathbb{F}}[\sigma])$
as products of the generators as in
Figure~\ref{fig:mcrank},
\ref{fig:fcrank}
and \ref{fig:midcrank}
respectively.
Here $\sigma$ is a word in the alphabet 
$\{s_1,\ldots, s_{|\pi(w)|-1}\}$.
\fig{fig:mcrank}{$C_{\mathbb{M}}[l]$}{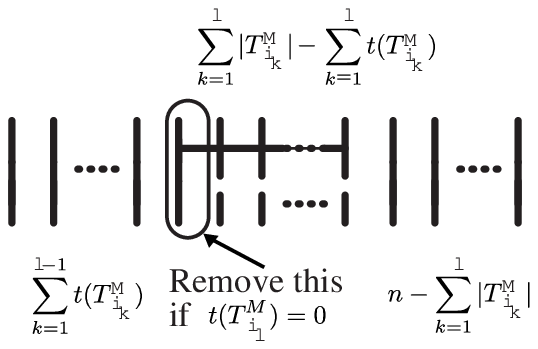}
\fig{fig:fcrank}{$C^*_{\mathbb{F}}[l]$}{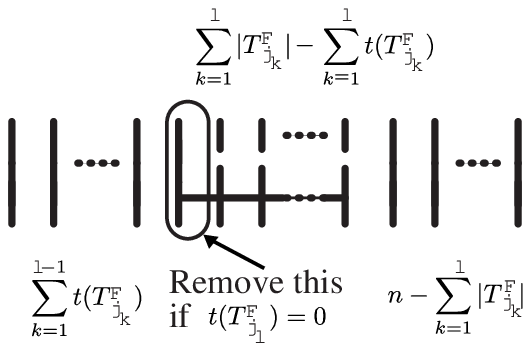}
\fig{fig:midcrank}{$C^{\mathbb{M}}_{\mathbb{F}}[\sigma]$}{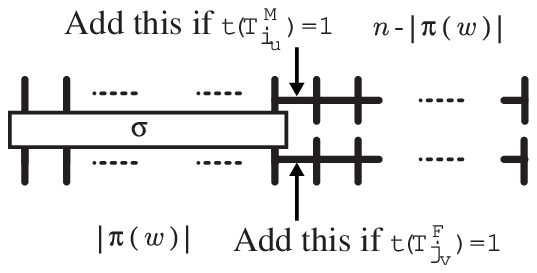}

Further we define the ``product of cranks''
${C}[\mathbb{M}]$ and  ${C}[\mathbb{F}]$
by
\[
	{C}[\mathbb{M}] = 
	C_{\mathbb{M}}[1]C_{\mathbb{M}}[2]\cdots C_{\mathbb{M}}[u-1]
\]
and
\[
	{C}^*[\mathbb{F}] =
	C^*_{\mathbb{F}}[v-1]\cdots C^*_{\mathbb{F}}[2]C^*_{\mathbb{F}}[1]
\]
respectively.
We note that $C_{\mathbb{M}}[l]$ [resp. $C^*_{\mathbb{F}}[l]$]
is defined by a composition $\mathbb{E} = (E_1,\ldots, E_s)$ of $n$
whose components have labels either ``propagating'' or ``defective''.
For example if $\mathbb{M} = (2,1,2,2,3)$,
$(t(M_i))_{1\leq i \leq 5} = (0,1,0,1,1)$,
$\mathbb{F} = (3,4,3)$, $(t(F_i))_{i=1,2,3} = (1,1,1)$
and $\sigma=(1,2)(2,3)\in\mathfrak{S}_3$,
then the product of cranks
${C}[\mathbb{M}]C^{\mathbb{M}}_{\mathbb{F}}[\sigma]C^*[\mathbb{F}]$
is presented as in Figure~\ref{fig:crank}.
\fig{fig:crank}{Product of cranks}{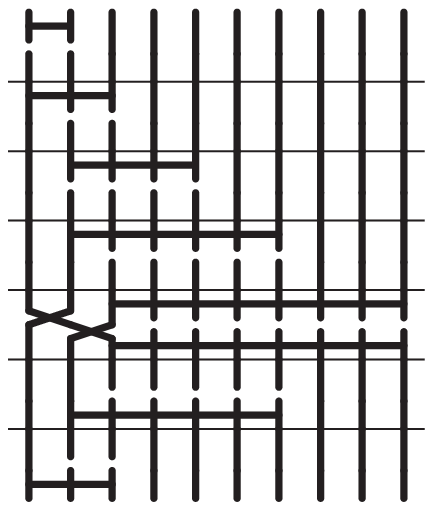}

Let $\overline{\mathbb M}$
be the sequence of $n$ symbols obtained from
$\mathbb{M} = \mathbb{M}(\rho_w)$ by
arranging all elements of $T^M_{i_k}$s
in accordance with the sequence $\mathbb{M}$
so that all elements of each $T^M_{i_k}$ are increasingly lined up
from left to right.
For example, if $\mathbb{M} = (\{3,1,7\},\{6,4\},\{5,2\})$,
then $\overline{\mathbb{M}} = (1,3,7,4,6,2,5)$.
Similarly $\overline{\mathbb{F}}$ is defined
from $\mathbb{F} = \mathbb{F}(\rho_w)$.

Then the following product becomes
an expression of a seat-plan $w$.
\[
	{\cal C}(\mathbb{M}, id, \mathbb{F})
	 = x_{\overline{\mathbb{M}}}{C}[\mathbb{M}]
	 C^{\mathbb{M}}_{\mathbb{F}}[id]
	 C^*[\mathbb{F}]x^*_{\overline{\mathbb{F}}}.
\]
Here $x_{\overline{\mathbb{M}}}$ [resp. $x^*_{\overline{\mathbb{F}}}$]
is a permutation
which maps $j$ to the number in the $j$-th coordinate
of $\overline{\mathbb{M}}$.
[resp. the number written in the $j$-th coordinate of
$\overline{\mathbb{F}}$ to $j'$].
We call this expression a {\em basic expression} of $w$.
We note that for a seat-plan $w$
there are several ways to choose $\rho_w$, a sequence of the parts of $w$.
{\em i.e.}
Several basic expressions can be defined for one seat-plan.

\subsection{The standard expressions of a seat-plan}
Our claim is that
any basic expression of a seat-plan $w$
can be moved to a special expression
called the {\em standard expression}
by using the relations $(R0)$-$(R4)$ and $(E1)$-$(E5)$ finite times.
In order to show this claim,
next we introduce the notion of a {\em crank form expression} of $w$.

Consider the propagating parts
$\pi(w)=\{T_{i_1},\ldots, T_{i_p}\}\ (p = |\pi(w)|)$ of $w$.
Let $(M_1, \ldots, M_p)$ be
a sequence of the upper parts of $\pi(w)$
and $(F_1, \ldots, F_p)$
the one of the lower parts.
Then there exists a permutation $\sigma\in{\mathfrak S_p}$
such that $\{M_{\sigma(k)}\sqcup F_{k}\ |\ k = 1, \ldots, p \} = \pi(w)$.
As is well known, a permutation $\sigma$ of degree $p$ is presented by
$p$-strings braid which connects the lower $k$-th point
to the upper $\sigma(k)$-th point.

Now we define a {\em crank form expression} of $w$.
Let $\mathbb{M} = (M_{1}, \ldots, M_{u})$
[resp. $\mathbb{F} = (F_1, \ldots, F_v)$]
be any fixed sequence of the upper [resp. lower] parts of $w$
(whose empty parts are omitted and propagating parts are specified).
From the sequences $\mathbb{M}$ and $\mathbb{F}$,
we obtain products of cranks $C[\mathbb{M}]$ and $C^*[\mathbb{F}]$.
Further,  from $\pi(\mathbb{M})$ and $\pi(\mathbb{F})$,
we obtain a permutation $\sigma\in{\mathfrak S}_p$
such that $\{M_{i_{\sigma(k)}}\sqcup F_{j_{k}}\ ;\ k = 1, \ldots, p \}
= \pi(w)$.
Then
the product
\[
	{\cal C}(\mathbb{M},\sigma,\mathbb{F})
	=
	x_{\mathbb{\overline{M}}}
	{C}[\mathbb{M}]C^{\mathbb{M}}_{\mathbb{F}}[\sigma]C^*[\mathbb{F}]
	x^*_{\overline{F}}
\]
becomes a presentation of $w$.
We call this presentation a {\em crank form expression of $w$
defined by ${\mathbb M}$ and ${\mathbb F}$}.
If a crank form expression is
made from sequences
$(M_1, \ldots, M_u)$ and $(F_1, \ldots, F_v)$ such that
\begin{enumerate}
\item
$M_1, \ldots, M_p$ and $F_1, \ldots, F_p$ are propagating,
\item
$M_{p+1}, \ldots, M_{u}$
and $F_{p+1}, \ldots, F_{v}$ are defective.
\end{enumerate}
then we call it {\em in normal form}.

Finally, we define the {\em standard expression} of $w$,
as a special expression of crank form expressions in normal form
by properly choosing the sequences $(M_1, \ldots, M_u)$
and $(F_1, \ldots, F_v)$.
For this purpose
first we sort the parts $T_1, \ldots, T_s$ of $w$ so that they satisfy:
\begin{enumerate}
\item
$\pi(w) = \{T_1, T_2, \ldots, T_p\},$
\item
$\{T_i\ |\ i = p+1, p+2, \ldots, u\}$
is the set of all upper defective parts,
\item
$\{T_i\ |\ i = u+1, u+2, \ldots, u+(v-p)\}$
is the set of all lower defective parts.
\end{enumerate}

For an ordered set $E$, let  $\min E$ be the minimum element in $E$.
Let $T_1, T_{2}, \ldots, T_{p}$ be the parts of $\pi(w)$.
Define $(M_1, M_{2}, \ldots, M_{p})$
so that they satisfy
\[
	\{M_1, M_{2}, \ldots, M_{p}\}
	= \{T_1^M, T_{2}^M, \ldots, T_{p}^M\}
\]
and
\[
	\min M_1 <\min M_{2} <\cdots <\min M_{p}.
\]
Similarly $(F_1, F_{2}, \ldots, F_{p})$
are defined using the lower parts of $\pi(w)$.
In such a method, the sequences of the upper parts $(M_1,\ldots, M_p)$ and
the lower parts $(F_1, \ldots, F_p)$ are uniquely defined
from a seat-plan $w$.

Now we define $(M_{p+1}, \ldots, M_{u})$
so that they satisfy
\[
	\{M_{p+1}, M_{p+2}, \ldots, M_{u}\}
	= \{T_{p+1}, T_{p+2}, \ldots, T_{u}\}
\]
and
\[
	\min M_{p+1} <\min M_{p+2} <\cdots <\min M_{u}.
\]
Similarly we define $(F_{p+1}, \ldots, F_{v})$
so that they satisfy
\[
	\{F_{p+1}, F_{p+2}, \ldots F_{v}\}
	= \{T_{u+1}, T_{u+2}, \ldots, T_{u+(v-p)}\}
\]
and
\[
	\min F_{p+1} <\min F_{p+2} <\cdots <\min F_{v}.
\]
Using these upper and lower sequences defined above,
we can obtain a crank from expression in normal form
called the {\em standard expression} of $w$.

\section{Proof of Theorem~1.2}

In the previous section,
we have defined the standard expression
of a word in the alphabet ${\cal L}_n^1$
as a special expression of the crank form expressions in normal form.
In this section, first we show that 
any two crank form expressions of a seat-plan $w$
are transformed to each other by finitely using the local moves
shown in Section~2.
Then we show
that any word in the alphabet ${\cal L}_n^1$
is moved to a scalar multiple of one of the crank form expressions.
Thus we can find that
any word in the alphabet ${\cal L}_n^1$
is reduced to a scalar multiple of a the standard expression.
Since the set of  seat-plans
makes a basis of $A_{n}(Q)$
and since every  seat-plan has its standard expression,
this proves that the partition algebra $A_{n}(Q)$ is characterized
by the generators and relations in Theorem~\ref{th:main}.

First we show that
any two crank form expressions
are transformed to each other.
For $w\in\Sigma_n^1$, let ${\mathbb M} = (M_1, \ldots, M_u)$
and ${\mathbb F} = (F_1, \ldots, F_v)$ be
sequences of the upper and the lower parts of $w$ respectively.
Assume that the subsequence $\pi(\mathbb{M})  = (M_{i_1}, \ldots, M_{i_p})$
($i_1<\cdots <i_p$) of $\mathbb{M}$
is the sequence of the upper propagating parts
and $\pi(\mathbb{F})  = (F_{j_1}, \ldots, F_{j_p})$
($j_1<\cdots <j_p$) is that of the lower propagating parts.
Then
there exists a permutation $\sigma$ of degree $p = |\pi(w)|$
which specifies how the propagating parts of $w$
are recovered from $\pi(\mathbb{M})$ and $\pi(\mathbb{F})$.
Let $\mathbb{E} = (E_1, \ldots, E_s)$
be a sequence of the upper or lower parts.
Suppose that $\tau\in{\mathfrak S}_s$ acts on $\mathbb{E}$
by $\tau\mathbb{E} = (E_{\tau^{-1}(1)}, \ldots, E_{\tau^{-1}(s)})$.
Then the following lemma holds.

\begin{lem}\label{lem:defect}
Let ${\mathbb M} = (M_1, \ldots, M_u)$
and ${\mathbb F} = (F_1, \ldots, F_v)$ be
sequences of the upper and the lower (non-empty) parts of
a seat-plan respectively.
If $M_i$ [resp. $F_i$] is defective
and $\sigma_i = (i,i+1)$, the $i$-th adjacent transposition,
then the crank form expression
${\cal C}(\mathbb{M}, \sigma, \mathbb{F})$ is moved to
another crank form expression
${\cal C}(\sigma_i\mathbb{M}, \sigma, \mathbb{F})$
[resp.
${\cal C}(\mathbb{M}, \sigma, \sigma_i\mathbb{F})$
].
\end{lem}

\begin{proof}
We consider the case $M_i$ is defective.
In case $F_i$ is defective, the similar proof will hold.
Let $P_{\mathbb{M},i}\in{\mathfrak S}_n$ be
a permutation defined by
\[
P_{\mathbb{M},i}(x):=
\left\{
	\begin{array}{ll}
		x + |M_{i+1}|&
\mbox{if}\ \sum_{j=1}^{i-1}|M_{j}| <x \leq \sum_{j=1}^{i}|M_{j}|,\\
		x - |M_{i}|&
\mbox{if}\ \sum_{j=1}^{i}|M_{j}| <x \leq \sum_{j=1}^{i+1}|M_{j}|,\\
		x&
		\mbox{otherwise}.
	\end{array}
\right.
\]
Then we find that
$x_{\overline{\mathbb{M}}}P^{-1}_{\mathbb{M},i}$
maps $j$ to the $j$-th coordinate of $\overline{\sigma_i\mathbb{M}}$.
Hence we have $x_{\overline{\mathbb{M}}}P^{-1}_{\mathbb{M},i}
 = x_{\overline{\sigma_i\mathbb{M}}}$.
(For the definition of $\overline{\mathbb{M}}$, see Section~3.2.)

On the other hand, since $M_i$ is defective,
we have $P_{\mathbb{M},i}C[\mathbb{M}] = C[\sigma_i\mathbb{M}]$
by removing an excrescence of $M_i$
and iteratively using ``defective part exchange'' ($R17'$)
in Figure~\ref{fig:dpeE} (if $M_{i+1}$ is defective)
or iteratively using ``defective part shift'' ($R16'$) 
in Figure~\ref{fig:dpsE} (if $M_{i+1}$ is propagating),
and then adding an excrescence to $M_i$ just moved.
Thus we obtain
\begin{eqnarray*}
{\cal C}(\mathbb{M}, \sigma, \mathbb{F})
 &=& x_{\overline{\mathbb{M}}}
	C[\mathbb{M}]C^{\mathbb{M}}_{\mathbb{F}}[\sigma]C^*[\mathbb{F}]
	x^*_{\overline{\mathbb{F}}}\\
 &=& (x_{\overline{\mathbb{M}}}P^{-1}_{\mathbb{M},i})
	(P_{\mathbb{M},i}C[\mathbb{M}])
	C^{\mathbb{M}}_{\mathbb{F}}[\sigma]C^*[\mathbb{F}]
	x^*_{\overline{\mathbb{F}}}\\
 &=& x_{\overline{\sigma_i\mathbb{M}}}
	C[\sigma_i\mathbb{M}]
	C^{\mathbb{M}}_{\mathbb{F}}[\sigma]C^*[\mathbb{F}]
	x^*_{\overline{\mathbb{F}}}\\
 &=&{\cal C}(\sigma_i\mathbb{M}, \sigma, \mathbb{F}).
\end{eqnarray*}
\end{proof}

\begin{rem}\label{rem:defect2}
Lemma~\ref{lem:defect} also holds
if $M_{i+1}$ [resp. $F_{i+1}$] is defective.
\end{rem}
By Lemma~\ref{lem:defect} and Remark~\ref{rem:defect2}
we may assume that any crank form expression
is given in normal form.

\begin{lem}\label{lem:crex}
Let ${\cal C}(\mathbb{M}, \sigma, \mathbb{F})$
be a crank form expression of $w$ in normal form.
If $M_i$ and $M_{i+1}$ are propagating
then ${\cal C}(\mathbb{M},\sigma,\mathbb{F})$
is moved to another crank form expression
${\cal C}(\sigma_i\mathbb{M}, \sigma_i\sigma,\mathbb{F})$ in normal form.
Similarly if $F_i$ and $F_{i+1}$ are propagating,
then  ${\cal C}(\mathbb{M},\sigma,\mathbb{F})$
is moved to  ${\cal C}(\mathbb{M},\sigma\sigma_i,\sigma_i\mathbb{F})$.
\end{lem}
\begin{proof}
Let $C_{\mathbb{M}}[i]$ and $C_{\mathbb{M}}[i+1]$
be $i$-th and $(i+1)$-st cranks of $C[\mathbb{M}]$.
By Figure~\ref{fig:crex}, we have
\[
P_{\mathbb{M},i}C_{\mathbb{M}}[i]C_{\mathbb{M}}[i+1]
=C_{\sigma_i\mathbb{M}}[i]C_{\sigma_i\mathbb{M}}[i+1]\sigma_i.
\]
\fig{fig:crex}{Crank form exchange}{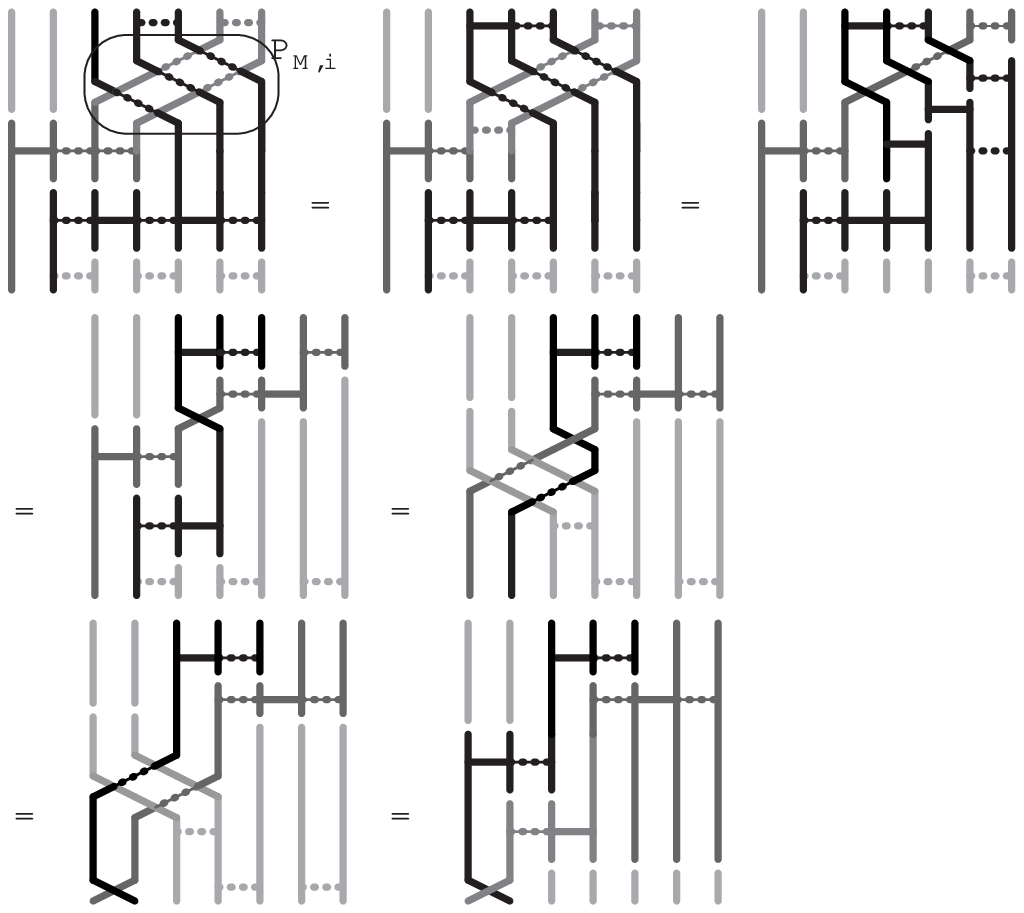}
Thus we obtain
\begin{eqnarray*}
{\cal C}(\mathbb{M}, \sigma, \mathbb{F})
&=&
x_{\overline{\mathbb{M}}}
	C[\mathbb{M}]
	C^{\mathbb{M}}_{\mathbb{F}}[\sigma]C^*[\mathbb{F}]
	x^*_{\overline{\mathbb{F}}}\\
&=&
(x_{\overline{\mathbb{M}}}P^{-1}_{\mathbb{M},i})
(P_{\mathbb{M},i}C[\mathbb{M}])
	C^{\mathbb{M}}_{\mathbb{F}}[\sigma]C^*[\mathbb{F}]
	y^*_{\overline{\mathbb{F}}}\\
&=&
x_{\overline{\sigma_i\mathbb{M}}}
C[\sigma_i\mathbb{M}]C^{\mathbb{M}}_{\mathbb{F}}[\sigma_i\sigma]
	C^*[\mathbb{F}]
	x^*_{\overline{\mathbb{F}}}\\
&=& {\cal C}(\sigma_i\mathbb{M}, \sigma_i\sigma, \mathbb{F}).
\end{eqnarray*}
\end{proof}

By Lemma~\ref{lem:defect}, Remark~\ref{rem:defect2} and Lemma~\ref{lem:crex}
we obtain the following.

\begin{prop}\label{prop:normalize}
A crank form expression of a seat-plan
is moved to its standard expression.
\end{prop}

Now we prove that any word in the alphabet ${\cal L}_n^1$
is moved to a crank form expression.
By the above proposition, we will find that
any word can be moved to its standard expression.

\begin{prop}
If ${\cal C}(\mathbb{M},\sigma,\mathbb{F})$ is the standard
expression of a seat-plan $w$,
then $s_i{\cal C}(\mathbb{M},\sigma,\mathbb{F})$ is 
moved to a crank form expression of $s_iw$.
\end{prop}

\begin{proof}
If $i$ and $i+1$ are both included one of the (upper) parts of $w$,
say $M_k$, then we have
\[
\sum_{j=1}^{k-1}|M_j|
<x^{-1}_{\overline{\mathbb{M}}}(i) < x^{-1}_{\overline{\mathbb{M}}}(i+1)
= x^{-1}_{\overline{\mathbb{M}}}(i)+1
\leq\sum_{j=1}^{k}|M_j|
\]
and
\[
(x^{-1}_{\overline{\mathbb{M}}}(i), x^{-1}_{\overline{\mathbb{M}}}(i+1))
C_{\mathbb{M}}[k]
= (x^{-1}_{\overline{\mathbb{M}}}(i), x^{-1}_{\overline{\mathbb{M}}}(i)+1)
C_{\mathbb{M}}[k]
= {\cal C}_{\mathbb{M}}[k].
\]
Since
\[
s_ix_{\overline{\mathbb{M}}} = (i,i+1)x_{\overline{\mathbb{M}}}
= x_{\overline{\mathbb{M}}}
(x^{-1}_{\overline{\mathbb{M}}}(i), x^{-1}_{\overline{\mathbb{M}}}(i+1)),
\]
we find that $s_i{\cal C}(\mathbb{M}, \sigma, \mathbb{F})
= {\cal C}(\mathbb{M}, \sigma, \mathbb{F})$ is a crank form expression.

If $i$ is included in $M_j$ and $i+1$ is included in $M_k$ ($j\neq k$),
then we have
$s_ix_{\overline{\mathbb{M}}} = x_{\overline{\mathbb{M}'}}$.
Here $\mathbb{M}'$ is the sequence of the upper parts
obtained from $\mathbb{M}=(M_1,\ldots, M_u)$
by replacing $M_j$ with $M_j' = (M_j\setminus\{i\})\cup\{i+1\}$
and $M_k$ with $M_k' = (M_k\setminus\{i+1\})\cup\{i\}$.

Hence we find that $s_i{\cal C}(\mathbb{M},\sigma,\mathbb{F})$
is moved to ${\cal C}(\mathbb{M}',\sigma,\mathbb{F})$,
a crank form expression.
In particular this expression again becomes
the standard expression, unless $k = j+1$,
$t(M_j) = t(M_{j+1})$, and $i = \min M_j$, $i+1 = \min M_{j+1}$.
\end{proof}

\begin{prop}
If ${\cal C}(\mathbb{M},\sigma,\mathbb{F})$ is the standard
expression of a seat-plan $w$,
then $f{\cal C}(\mathbb{M},\sigma,\mathbb{F})$ is 
moved to a crank form expression of $fw$.
\end{prop}

\begin{proof}
First consider the case $\{1, 2\}\subset M_k$ for some $k$.
In this case, there exists an integer $i$
such that $i = x^{-1}_{\overline{\mathbb{M}}}(1)$
and $i+1 = x^{-1}_{\overline{\mathbb{M}}}(2)$. 
Hence in this case we have
$fx_{\overline{\mathbb{M}}} = x_{\overline{\mathbb{M}}}f_i$
and
$f_i{\cal C}_{\mathbb{M}}[k] = {\cal C}_{\mathbb{M}}[k]$.
Thus we obtain $f{\cal C}(\mathbb{M}, \sigma, \mathbb{F})
= {\cal C}(\mathbb{M}, \sigma, \mathbb{F})$.

Next consider the case $1\in M_j$ and $2\in M_k$ ($j\neq k$).
In the following we assume that $M_j$ and $M_k$ are both propagating.
Even if either $M_j$ or $M_k$ or both of them are defective,
the similar proof will hold.
Proposition~\ref{prop:normalize} implies that
the standard expression ${\cal C}(\mathbb{M}, \sigma, \mathbb{F})$
is moved to a crank form expression
${\cal C}(\mathbb{M}', id, \mathbb{F}')$
so that the first
and the second components of $\mathbb{M}'$
are $M_j$ and $M_k$ respectively
and the first and the second components of $\mathbb{F}'$
are jointed to $M_j$ and $M_k$ respectively.
Using the relations ($R2''$), ($R2$)
and ($R12''$),
we find that the first and the second components of
$\mathbb{M}'$ and those of $\mathbb{F}'$
are merged by the action of $f$.
For example, if $|M_j| = 5$ and $|M_k| = 4$
then we have Figure~\ref{fig:fwE}.
\fig{fig:fwE}{Action of $f$ on $w$}{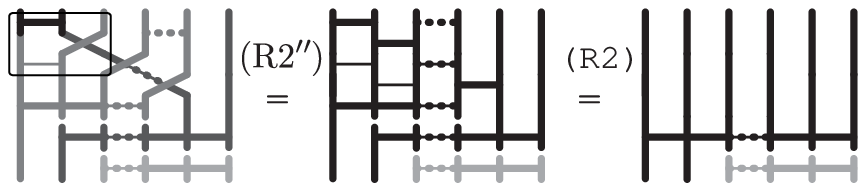}
The merged propagating parts will be moved to a
crank form expression ${\cal C}(\mathbb{M}'', id, \mathbb{F}'')$
by ``bumping'' as in Figure~\ref{fig:bump}.
Here $\mathbb{M}''$ [resp. $\mathbb{F}''$]
is a sequence of upper [resp. lower] parts
obtained from $\mathbb{M}$ [resp. $\mathbb{F}$]
by merging the first two components.
\fig{fig:bump}{Bumping}{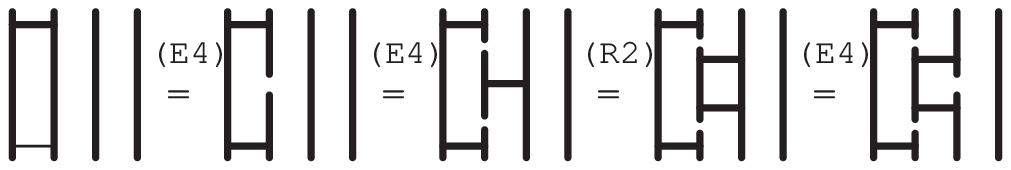}
\end{proof}

\begin{prop}
If ${\cal C}(\mathbb{M},\sigma,\mathbb{F})$ is the standard
expression of a seat-plan $w$,
then $e{\cal C}(\mathbb{M},\sigma,\mathbb{F})$ is 
moved to a crank form expression of $ew$.
\end{prop}

\begin{proof}
By the same argument in the previous proposition,
we may assume that ${\cal C}(\mathbb{M},\sigma,\mathbb{F})$
is moved to a crank form expression
\[
	{\cal C}(\mathbb{M}'', id, \mathbb{F}'')
\]
such that the first component $M_1''$ of $\mathbb{M}''$
contains $\{1\}$.

First consider the case $|M_1''|>1$.
In this case, it is easy to check that
$e{\cal C}(\mathbb{M}'', id, \mathbb{F}'')$
is again a crank form expression
of $ew$ as it is.

Next consider the case $|M_1''| = 1$.
If $M_1''$ is defective,
then we have  a scalar multiple of a crank form expression
$e{\cal C}(\mathbb{M}'', id, \mathbb{F}'') =
Q{\cal C}(\mathbb{M}'', id, \mathbb{F}'')$.
If $M_1''$ is propagating,
then applying ``addition of excrescences ($R14'$)''
and ``bumping''
in Figures~\ref{fig:roeE} and \ref{fig:bump}
we can move $e{\cal C}(\mathbb{M}'', id, \mathbb{F}'')$ to a crank form
expression.
\end{proof}

\begin{proof}[Proof of Theorem~1.2]
Let $\widetilde{A_{n}(Q)}$
be the associative algebra over $\mathbb{Z}[Q]$
abstractly defined by the generators and the relations
in Theorem~1.2.
So there exists a surjective morphism $\psi$
from $\widetilde{A_{n}(Q)}$ to $A_{n}(Q)$.
As we noted in Section~1,
we may assume that $\widetilde{A_{n}(Q)}$
is generated by the alphabets ${\cal L}_n^1$
which satisfy the relations ($R0$)-($R4$) and ($E1$)-($E5$).
Here we note that the ``geometrical moves''
we have shown previously
can be applied to any algebra which satisfies
the relations ($R0$)-($R4$) and ($E1$)-($E5$).
Hence if we associate the alphabets in ${\cal L}_n^1$
with the diagrams in Figure~\ref{fig:gen},
then we can apply the notion of {\em basic expressions},
{\em crank form expressions}
and {\em standard expressions} to the words in the alphabets
${\cal L}_n^1$ of $\widetilde{A_{n}(Q)}$.

Let $w$ be a word in the alphabet ${\cal L}_n^1$ of $\widetilde{A_{n}(Q)}$.
Suppose that $w$ is presented in a standard expression.
Then by Proposition~4.5-4.7,
$s_iw$, $fw$ and $ew$ are all moved to
(possibly scalar multiples of) crank form expressions.
By Proposition~4.4, they are moved to the standard expressions.
Since $s_i$ ($1\leq i\leq n-1$), $f$, and $e$ are
crank form expressions as they are,
by induction on the lengths of the words in the alphabets ${\cal L}_n$,
any word turn out to be equal to (a scalar multiple) of
the standard expression of a seat-plan $w$ of $\Sigma_n^1$.
Hence we have
\[
	\mbox{rank}\ \widetilde{A_{n}(Q)} \leq |\Sigma_n^1|.
\]

As Tanabe showed in \cite{Ta},
$\Sigma_n^1$ makes a basis of
$\Comp\otimes A_{n}(k) = \Comp\otimes \psi(\widetilde{A_{n}(k)})$
if $k\geq n$.
Hence $\mbox{rank}\ \Comp\otimes A_{n}(z) = |\Sigma_n^1|$
holds as far as $z$ takes any integer value $k\geq n$. 
This implies that $\psi$ is an isomorphism
and we find that the generators and the relations
in Theorem~1.2 characterize the partition algebra $A_{n}(Q)$.
\end{proof}

\section{Definition of $A_{n-\frac{1}{2}}(Q)$, a subalgebra of $A_n(Q)$}
\label{sec:5-1}
In this section,
we consider a subalgebra $A_{n-\frac{1}{2}}(Q)$ of $A_n(Q)$
generated by the special elements
$s_1, \ldots, s_{n-2}$, $f_1, \ldots, f_{n-1}$
and $e_1,\ldots, e_{n-1}$.
As we have noted in Remark~\ref{rem:gen},
$\{f_i\}$ ($1\leq n-2$)
and $\{e_i\}$ ($1\leq n-1$)
are written as products of $f=f_1$, $e = e_1$ and $s_1,\ldots, s_{n-2}$.
The special element $f_{n-1}$, however,
can not be expressed
as a product of other special elements in $A_{n-\frac{1}{2}}(Q)$,
since we deleted $s_{n-1}$ from the generators of $A_n(Q)$.
Hence $A_{n-\frac{1}{2}}(Q)$ can be defined
as a subalgebra of $A_n(Q)$
generated by the following elements:
$s_1, \ldots, s_{n-2}$, $f = f_1$, $f_* = f_{n-1}$
and $e = e_1$.
We can obtain the defining relations
among these generators just as in the case of $A_n(Q)$.

\begin{thm}\label{def:half-int-alg}
Let $\mathbb{Z}$
be the ring of rational integers
and $Q$ the indeterminate.
We put ${A}_{\frac{1}{2}}(Q) = \mathbb{Z}[Q]\cdot 1$.
For an integer $n\geq2$,
${A}_{n-\frac{1}{2}}(Q)$ is characterized by the
generators
\[
	e, f, s_1, s_2, \ldots, s_{n-2}, f_{*} \mbox{(if $n>2$)}
\]
and the relations
($R0$), ($R1'$)-($R4'$) and ($E1'$)-($E5'$) omitting
the ones which involve $s_{n-1}$ and adding the
following relations:
\begin{gather} 
f_{*}s_{n-2}s_{n-3}\cdots s_3s_2s_1
s_2s_3\cdots s_{n-3}s_{n-2}f_{*} \nonumber\\
	\quad =\, f_{*}s_{n-2}s_{n-3}\cdots s_3s_2f
	s_2s_3\cdots s_{n-3}s_{n-2}\phantom{,} \tag{$R2^*$} \\
	\quad =\, s_{n-2}s_{n-3}\cdots s_3s_2f
	s_2s_3\cdots s_{n-3}s_{n-2}f_{*},\nonumber
\end{gather}
\begin{equation}\tag{$R4^*$}
ff_{*} = f_{*}f,\quad 
ef_{*} = f_{*}e,\quad
f_{*}s_i = s_if_{*}\ \mbox{($1\leq i \leq n-3$)},
\end{equation}
\begin{equation}\tag{$E4^*$}
	\begin{array}{rcl}
	f_{*}s_{n-2}s_{n-3}\cdots s_1es_1\cdots s_{n-3}s_{n-2}f_{*}
	&=& f_{*},\\
	es_{1}s_{2}\cdots s_{n-2}f_{*}s_{n-2}\cdots s_{2}s_{1}e
	&=& e.
	\end{array}
\end{equation}
We understand $A_{1+\frac{1}{2}}(Q) = A_{2-\frac{1}{2}}(Q)$
is defined by the generators $1$, $e$ and $f$ with the relations
$e^2 = Qe$, $f^2 = f$, $efe =e$, $fef = f$.
(Hence, $A_{2-\frac{1}{2}}(Q)$ is a rank 5 module with a basis
$\{1, e, f, ef, fe\}$.)
\end{thm}
The relations ($R2^*$) correspond to the relations
$f_{n-1}s_{n-2}f_{n-1} = f_{n-1}f_{n-2} = f_{n-2}f_{n-1}$.
We deduce $f_{*}s_{n-2}f_{*} = f_{*}s_{n-2}f_{*}s_{n-2}
= f_{*}s_{n-2}f_{*}s_{n-2}$ from ($R2^*$).
\begin{proof}
First we note that all the generators of $A_{n-\frac{1}{2}}(Q)$
have the part which contains $n$ and $n'$ simultaneously.

We consider the transpositions of indices
$i\leftrightarrow n-i+1$.
These transpositions make
$A_{n-\frac{1}{2}}(Q)$ a subalgebra of $A_{n}(Q)$
generated by
\[
	{\cal L}_{n-\frac{1}{2}}^1 = \{f_1, \ldots, f_{n-1},
	s_2, \ldots, s_{n-1}, e_2, \ldots, e_n\}.
\]
By the relation ($R0$), $A_{n-\frac{1}{2}}(Q)$ is actually
generated by letters
$\{f_1$, $f_2$, $e_2$ and $s_2,\ldots, s_{n-1}\}$.
Each of these generators has a part which includes $\{1, 1'\}$.
In the following in this section, we suppose that $A_{n-\frac{1}{2}}(Q)$
is generated by the letters in ${\cal L}^1_{n-\frac{1}{2}}$.
The $\mathbb{Z}[Q]$ bases of $A_{n-\frac{1}{2}}(Q)$
consist of $\Sigma^1_{n-\frac{1}{2}}$
a subset of seat-plans in $\Sigma^1_n$
which have at least one propagating part
which contains $1$ and $1'$ simultaneously.
In the diagram of the standard expression of a seat-plan of
$\Sigma^1_{n-\frac{1}{2}}$,
the vertices $1$ and $1'$ are joined by a  vertical line.
Shrinking this vertical line to one vertex,
we have one to one correspondences between $\Sigma^1_{n-\frac{1}{2}}$
and
the set of the set-partitions of order $2n-1$.
(Hence we find $|\Sigma^1_{n-\frac{1}{2}}| = B_{2n-1}$, the Bell number.)

Under this preparation, we prove the theorem.
Since the relations in the theorem allow us to use
all the required local moves, we can show
just in the course of the arguments of Section~4
that any word in the alphabet ${\cal L}^1_{n-\frac{1}{2}}$
is equal to (possibly a scalar multiple of) a standard expression
in the abstract algebra $\widetilde{A_{n-\frac{1}{2}}(Q)}$.

Hence we have
\[
	\mbox{rank}\ \widetilde{A_{n-\frac{1}{2}}}(Q)
	\leq |\Sigma_{n-\frac{1}{2}}^1|.
\]

As Murtin and Rollet showed in \cite{MR},
$\Sigma_{n-\frac{1}{2}}^1$ makes a basis of
$\Comp\otimes A_{n-\frac{1}{2}}(k)
= \Comp\otimes \psi(\widetilde{A_{n-\frac{1}{2}}(k)})$
if $k>n$.
Hence $\mbox{rank}\ \Comp\otimes A_{n}(z) = |\Sigma_{n-\frac{1}{2}}^1|$
holds as far as $z$ takes any integer value $k> n$. 
This implies that $\psi$ is an isomorphism
and we find that the generators and the relations
in the theorem characterize the subalgebra
$A_{n-\frac{1}{2}}(Q)$.
\end{proof}

\section{Bratteli diagram of the partition algebras}\label{sec:bra}

In this section, we get back to the original definition
of $A_{n-\frac{1}{2}}(Q)$.
({\it i. e.} $A_{n-\frac{1}{2}}(Q)$ is
generated by $s_1, \ldots, s_{n-2}$, $f_1, \ldots, f_{n-1}$
and $e_1,\ldots, e_{n-1}$.)
Since, $A_{n-\frac{1}{2}}(Q)$ contains all the generators of
${A}_{n-1}(Q)$, it becomes
a subalgebra of ${A}_{n-\frac{1}{2}}(Q)$.
Hence we obtain the sequence of inclusions
$A_0(Q)\subset A_\frac{1}{2}(Q) \subset \cdots \subset A_{i-\frac{1}{2}}(Q)
\subset A_{i}(Q)\subset A_{i+\frac{1}{2}}(Q)\subset \cdots$.

First we define a graph $\Gamma_n$ [resp. $\Gamma_{n+\frac{1}{2}}$]
for a non-negative integer $n\in\mathbb{Z}_{\geq 0}$.
Then we define the sets of {\em tableaux} as sets of paths on this graph.
Figure~\ref{fig:brad} will help the reader to understand
the recipe.
\fig{fig:brad}{$\Gamma_4$}{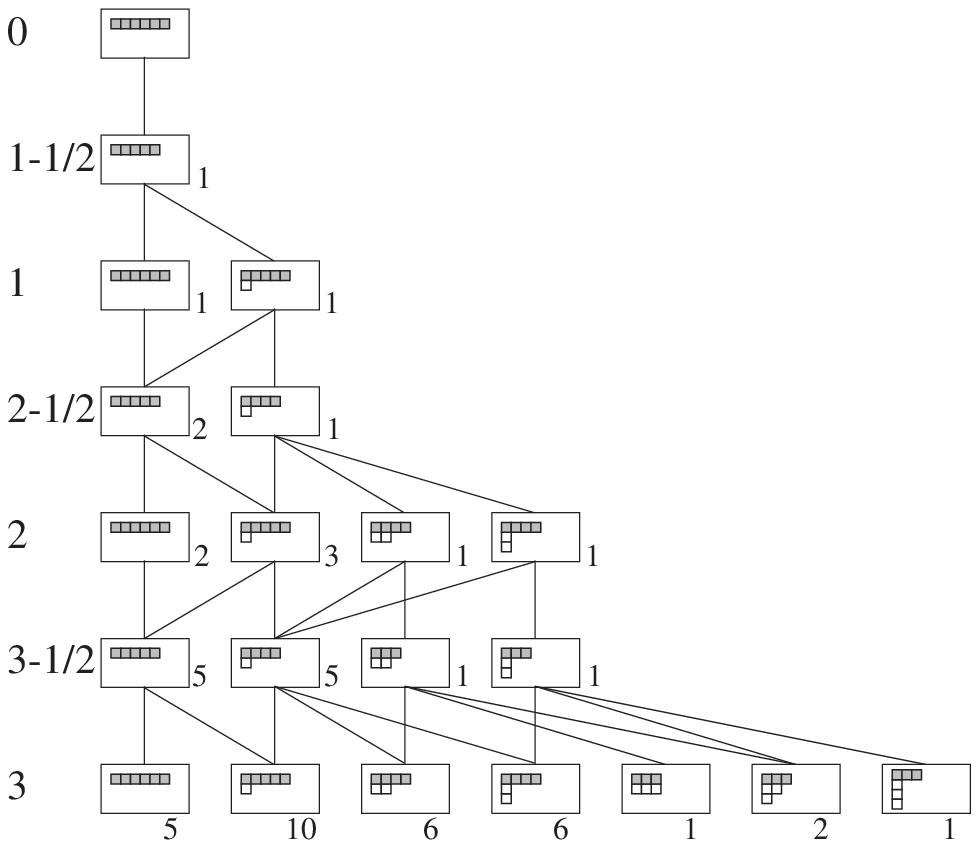}

For the moment, we assume that $Q$ is a sufficiently large integer.
Let
$\lambda = (\lambda_1, \lambda_2, \ldots, \lambda_l)$
be a partition.
For this $\lambda$,
define 
\begin{eqnarray*}
\widetilde{\lambda}
&=& (Q-|\lambda|, \lambda_1, \lambda_2, \ldots, \lambda_l)\\
\big[\mbox{resp.}\ \widehat{\lambda}
&=& (Q-1-|\lambda|, \lambda_1, \lambda_2, \ldots, \lambda_l)\big]
\end{eqnarray*}
to be a partition of size $Q$ [resp. $Q-1$].
Pictorially,
$\widetilde{\lambda}$
[resp. $\widehat{\lambda}$]
is obtained by adding $Q-|\lambda|$
[resp. $Q-1-|\lambda|$] boxes on the top of $\lambda$.

Let
$
P_{\leq i}
= \bigcup_{j=0}^i\{ \lambda\ |\ \lambda\vdash j\}
$
be a set of Young diagrams of size less than or equal to $i$.
We define $\bfLambda_i$ and $\bfLambda_{i+\frac{1}{2}}$ to be
\[
\bfLambda_i = \{\widetilde{\lambda}\ |\ \lambda\in P_{\leq i}\}
\mbox{ and } 
\bfLambda_{i+\frac{1}{2}} = \{\widehat{\lambda}\ |\ \lambda\in P_{\leq i}\},
\]
which are set of Young diagrams of size $Q$ and $Q-1$ respectively.

Under these preparations
we define a graph $\Gamma_n$ [resp. $\Gamma_{n+\frac{1}{2}}$]
which consists of the vertices labeled by:
\[
\left(
	\bigsqcup_{i=0,1, \ldots, n-1}
	(\bfLambda_i \sqcup \bfLambda_{i+\frac{1}{2}})
\right)
\bigsqcup\bfLambda_n
\quad
\left[	\mbox{resp.}\  
\left(
	\bigsqcup_{i=0,1, \ldots, n}
	(\bfLambda_i \sqcup \bfLambda_{i+\frac{1}{2}})
\right)
\right]
\]
and the edges joined by either of the following rule:
\begin{itemize}
\item join $\widetilde{\lambda}\in\bfLambda_{i}$
and $\widehat{\mu}\in\bfLambda_{i+\frac{1}{2}}$ if
$\widehat{\mu}$ is obtained from $\widetilde{\lambda}$
by removing a box ($i = 0, 1, 2, \ldots n-1$)
[resp. ($i=0, 1, 2, \dots, n$)],
\item
join $\widehat{\mu}\in\bfLambda_{i-\frac{1}{2}}$
and $\widetilde{\lambda}\in\bfLambda_i$ if
$\widetilde{\lambda}$ is obtained from $\widehat{\mu}$
by adding a box ($i = 1, 2, \ldots n$).
\end{itemize}
For a pair of Young diagrams $(\bfalpha, \bfbeta)$,
if $\bfbeta$ is obtained from $\bfalpha$ by one of the
method above, we write this as $\bfalpha\smile\bfbeta$.

Finally, we define the sets of the tableaux.
For a half integer $n\in\frac{1}{2}\mathbb{Z}$
and $\bfalpha\in\bfLambda_n$,
we define ${\mathbb T}(\bfalpha)$, {\em tableaux of shape $\bfalpha$},
to be
\begin{eqnarray*}
{\mathbb T}(\bfalpha)&=&
	\{P = (\bfalpha^{(0)}, \bfalpha^{(1/2)}, \ldots, \bfalpha^{(n)})\ |
	\ \bfalpha^{(j)} \in \bfLambda_j\ (j = 0, 1/2, \ldots, n),\\
& &	\quad\bfalpha^{(n)} = \bfalpha, 
	\bfalpha^{(j)}\smile\bfalpha^{(j+1/2)}
	\ (j = 0, 1/2, \ldots, n-1/2)\}.
\end{eqnarray*}


\section{Construction of representation}\label{sec:rep}

Now we have defined the sets of tableaux,
we define linear transformations among the tableaux.

Let ${\mathbb Q}$ be the field of rational numbers
and $K_0 = {\mathbb Q}(Q)$ its extension.
In the following, the linear transformations are defined over $K_0$.
If they preserve the relations defined in the previous sections,
they define representations
of ${A}_n = {A}_n(Q)\otimes K_0$.
Similar methods are used for example
in the references~\cite{AK,GHJ,Mu,W1,W2,Ko2}.

Let ${\mathbb V}(\bfalpha)
= \oplus_{P \in {\mathbb T}(\bfalpha) }K_0 v_P$ be
a vector space over $K_0$ with the standard basis
$\{v_P|P\in {\mathbb T}(\bfalpha)\}$.

For generators $e_i$, $f_i$ and $s_i$ of ${A}_n$,
we define linear maps $\rho_{\bfalpha}(e_i)$, $\rho_{\bfalpha}(f_i)$
and $\rho_{\bfalpha}(s_i)$
on ${\mathbb V}(\bfalpha)$ giving the matrices
$E_i$ $F_i$ and $M_i$ respectively with respect to
the basis $\{ v_P | P\in {\mathbb T}(\bfalpha) \}$.

\subsection{Definition of $\rho_{\bfalpha}(e_i)$}
Firstly, we define a linear map for $e_i$.

For a tableaux
$P = (\bfalpha^{(0)}, \bfalpha^{(1/2)}, \ldots, \bfalpha^{(n)})$
of ${\mathbb T}(\bfalpha)$, we define
$\rho_{\bfalpha}(e_i)(v_P)
= \sum_{Q \in {\mathbb T}(\bfalpha)}(E_i)_{QP}v_Q$.
Let $Q = (\bfalpha^{\prime(0)}, \bfalpha^{\prime(1/2)},
\ldots, \bfalpha^{\prime(n)})$.

If there is an
$i_0 \in \{1/2, 1, \ldots, n-1/2 \} \setminus \{i-1/2\}$
such that $\bfalpha^{(i_0)}\neq \bfalpha^{\prime(i_0)}$,
then we put
\[
	(E_i)_{QP} = 0.
\]
In the following, we consider the case that
$\bfalpha^{(i_0)} = \bfalpha^{\prime(i_0)}$
for $i_0\in\{0, 1/2, 1, \ldots, n-1/2\}\setminus\{i-1/2\}$.

If
$\bfalpha^{(i-1)}$ and $\bfalpha^{(i)}$
are not labeled by the same Young diagram,
then we put
\[
	(E_i)_{QP} = 0.
\]

We consider the case $\bfalpha^{(i-1)}$ and $\bfalpha^{(i)}$
have the same label $\widetilde{\lambda}$.
In this case, the possible vertices as $\bfalpha^{(i-1/2)}$
have labels $\{\widetilde{\lambda}^{-}_{(s)}\}$,
which are obtained by removing one box from $\widetilde{\lambda}$.
Let $\{Q_s\}$ be the set of tableaux
obtained from $P$ by replacing $\bfalpha^{(i-1/2)}$
with $\widetilde{\lambda}^{-}_{(s)}$.

Then we define $(E_i)_{QP}$ to be
\[
	(E_i)_{Q_sP}
	= \frac{h(\widetilde{\lambda})}{h(\widetilde{\lambda}^{-}_{(s)})}.
\]
Here $h(\lambda)$ is the product of hook lengths defined by
\[
	h(\lambda)  = \prod_{x\in\lambda} h_{\lambda}(x)
\]
and $h_{\lambda}(x)$ is the {\em hook-length} at $x\in\lambda$.

Note that the matrix $E_i$ is determined
by
the label $\widetilde{\lambda}$ itself
not by the vertex at which the tableau $P$ goes through.
In other words, if another vertex in different level, say $i'$,
has the same label $\widetilde{\lambda}$, then $E_{i'}$ becomes
the same matrix.

Let $v(\lambda^-_{(s)}, \lambda)$ be the standard vector
which corresponds to a tableau whose $(i-1)$-st, $(i-1/2)$-th and $i$-th
coordinate $(\bfalpha^{(i-1)}, \bfalpha^{(i-1/2)}, \bfalpha^{(i)})$
are labeled by $(\lambda, \lambda^-_{(s)}, \lambda)$.
Then for a tableau $P$ which goes through
$\lambda$ at the $(i-1)$-st and the $i$-th coordinate of $P$,
we have
\[
\rho(e_i)(v_P)
=
\sum_{s'} \frac{h(\lambda)}{h(\lambda^{-}_{(s')})}v(\lambda^-_{(s')}, \lambda).
\]
Here $\lambda^{-}_{(s')}$ runs through Young diagrams
obtained from $\lambda$ by removing one box.

\fig{fig:repE4a}{Representation spaces for $\rho(e_i)$}{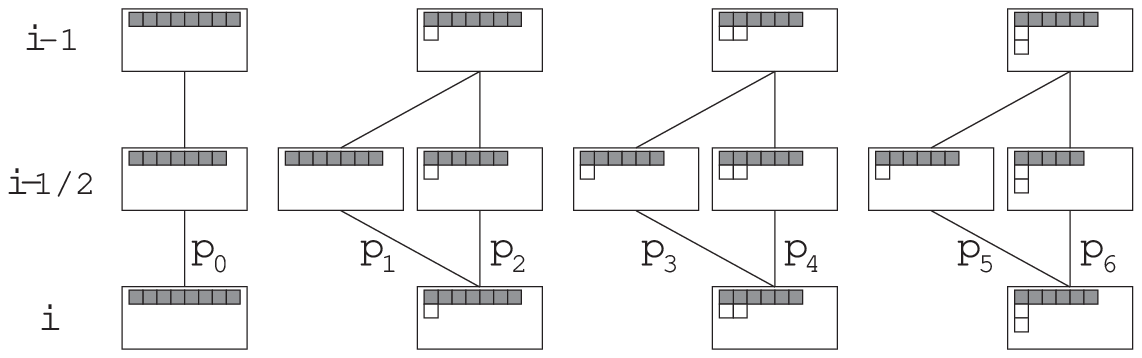}
\fig{fig:repE4b}{Representation spaces for $\rho(e_i)$}{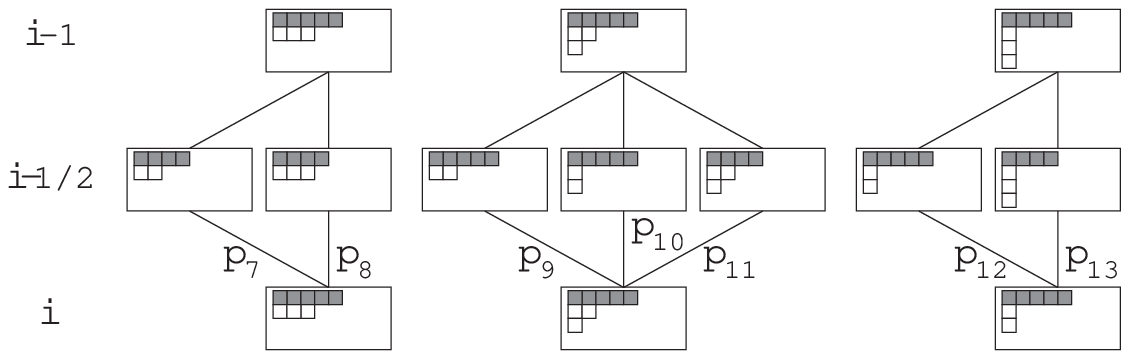}

\begin{example}
Suppose that tableaux $\{p_r\}$ goes through
paths in pictures illustrated in Figure~\ref{fig:repE4a}
or \ref{fig:repE4b}.
Then we have
\begin{eqnarray*}
\rho(e_i)(v_0) &=&
\frac{h(\widetilde{\emptyset})}{h(\widehat{\emptyset})}v_0 = Qv_0,\\
\rho(e_i)(v_1\ v_2) &=& (v_1\ v_2)
\begin{pmatrix}
	h(\widetilde{\abox})/h(\widehat{\emptyset})
	&h(\widetilde{\abox})/h(\widehat{\emptyset})\\
	h(\widetilde{\abox})/h(\widehat{\abox})
	&h(\widetilde{\abox})/h(\widehat{\abox})
\end{pmatrix}\\
&=& (v_1\ v_2)
\begin{pmatrix}
	\frac{Q}{Q-1} &\frac{Q}{Q-1}\\
	\frac{Q(Q-2)}{Q-1} &\frac{Q(Q-2)}{Q-1}
\end{pmatrix}
\\
\rho(e_i)(v_3\ v_4) &=& (v_3\ v_4)
\begin{pmatrix}
	h(\widetilde{\horibox})/h(\widehat{\abox})
	&h(\widetilde{\horibox})/h(\widehat{\abox})\\
	h(\widetilde{\horibox})/h(\widehat{\horibox})
	&h(\widetilde{\horibox})/h(\widehat{\horibox})
\end{pmatrix},\\
&=& (v_3\ v_4)
\begin{pmatrix}
	\frac{2(Q-2)}{Q-3} &\frac{2(Q-2)}{Q-3}\\
	\frac{(Q-1)(Q-4)}{Q-3} &\frac{(Q-1)(Q-4)}{Q-3}
\end{pmatrix}\\
\rho(e_i)(v_5\ v_6) &=& (v_5\ v_6)
\begin{pmatrix}
	h(\widetilde{\vertbox})/h(\widehat{\abox})
	&h(\widetilde{\vertbox})/h(\widehat{\abox})\\
	h(\widetilde{\vertbox})/h(\widehat{\vertbox})
	&h(\widetilde{\vertbox})/h(\widehat{\vertbox})
\end{pmatrix},\\
 &=& (v_5\ v_6)
\begin{pmatrix}
	\frac{2Q}{Q-1} &\frac{2Q}{Q-1}\\
	\frac{Q(Q-3)}{Q-1} &\frac{Q(Q-3)}{Q-1}
\end{pmatrix}.
\end{eqnarray*}
Here $v_i$ is the standard vector which corresponds to $p_i$.
Similarly for the bases $\langle v_7, v_8\rangle$,
$\langle v_9, v_{10}, v_{11}\rangle$
and
$\langle v_{12}, v_{13}\rangle$,
we have the following matrices respectively:
\begin{eqnarray*}
&&\begin{pmatrix}
	\frac{3(Q-4)}{Q-5} &\frac{3(Q-4)}{Q-5}\\
	\frac{(Q-2)(Q-6)}{Q-5} &\frac{(Q-2)(Q-6)}{Q-5}
\end{pmatrix},\\
&&
\begin{pmatrix}
	\frac{3(Q-1)}{2(Q-2)} &\frac{3(Q-1)}{2(Q-2)} &\frac{3(Q-1)}{2(Q-2)}\\
	\frac{3(Q-3)}{2(Q-4)} &\frac{3(Q-3)}{2(Q-4)} &\frac{3(Q-3)}{2(Q-4)}\\
	\frac{(Q-1)(Q-3)(Q-5)}{(Q-2)(Q-4)}&
	\frac{(Q-1)(Q-3)(Q-5)}{(Q-2)(Q-4)}&
	\frac{(Q-1)(Q-3)(Q-5)}{(Q-2)(Q-4)}
\end{pmatrix},\\
&&
\begin{pmatrix}
	\frac{3Q}{Q-1} &\frac{3Q}{Q-1}\\
	\frac{Q(Q-4)}{Q-1} &\frac{Q(Q-4)}{Q-1}
\end{pmatrix}.
\end{eqnarray*}
\end{example}

\subsection{Definition of $\rho_{\bfalpha}(f_i)$}
Next, we define a linear map for $f_i$.

For a tableaux
$P = (\bfalpha^{(0)}, \bfalpha^{(1/2)}, \ldots, \bfalpha^{(n)})$
of ${\mathbb T}(\bfalpha)$, we define
$\rho_{\bfalpha}(f_i)(v_P)
= \sum_{Q \in {\mathbb T}(\bfalpha)}(F_i)_{QP}v_Q$.
Let $Q = (\bfalpha^{\prime(0)}, \bfalpha^{\prime(1/2)},
\ldots, \bfalpha^{\prime(n)})$.

If there is an
$i_0 \in \{1/2, 1, \ldots, n-1/2 \} \setminus \{i\}$
such that $\bfalpha^{(i_0)}\neq \bfalpha^{\prime(i_0)}$,
then we put
\[
	(F_i)_{QP} = 0.
\]
In the following, we consider the case that
$\bfalpha^{(i_0)} = \bfalpha^{\prime(i_0)}$
for $i_0\in\{0, 1/2, 1, \ldots, n-1/2\}\setminus\{i\}$.

If
$\bfalpha^{(i-1/2)}$ and $\bfalpha^{(i+1/2)}$
are not labeled by the same Young diagram,
then we put
\[
	(F_i)_{QP} = 0.
\]

We consider the case $\bfalpha^{(i-1/2)}$ and $\bfalpha^{(i+1/2)}$
have the same label $\widehat{\mu}$.
In this case, the possible vertices as $\bfalpha^{(i)}$
have labels $\{\widehat{\mu}^{+}_{(r)}\}$,
which are obtained by adding one box to $\widetilde{\mu}$.
Suppose that $\bfalpha^{(i)}$,
the $i$-th coordinate of $P$,
has its label $\widetilde{\mu}^{+}_{(r_0)}$.
Let $Q$ be a tableau
obtained from $P$ by replacing $\bfalpha^{(i)}$
with one of $\{\widehat{\mu}^{+}_{(r)}\}$.

Then we define $(F_i)_{QP}$ to be
\[
	(F_i)_{Q_rP} = \frac{h(\widehat{\mu})}{h(\widehat{\mu}^{+}_{(r_0)})}.
\]

Let $v(\mu^+_{(r)}, \mu)$ be the standard vector
which corresponds to a tableau whose $(i-1/2)$-th, $i$-th and $(i+1/2)$-th
coordinate $(\bfalpha^{(i-1/2)}, \bfalpha^{(i)}, \bfalpha^{(i+1/2)})$
are labeled by $(\mu, \mu^+_{(r)}, \mu)$.
Then for a tableau $P$ which goes through
$\mu$ at the $(i-1/2)$-th and the $(i+1/2)$-th coordinate of $P$,
we have
\[
\rho(f_i)(v_P)
=
\sum_{r} \frac{h(\mu)}{h(\mu^{+}_{(r_0)})}v(\mu^+_{(r)}, \mu).
\]
Here $\mu^{+}_{(r)}$ runs through Young diagrams
obtained from $\mu$ by adding one box.

\fig{fig:repF4}{Representation spaces for $\rho(f_i)$}{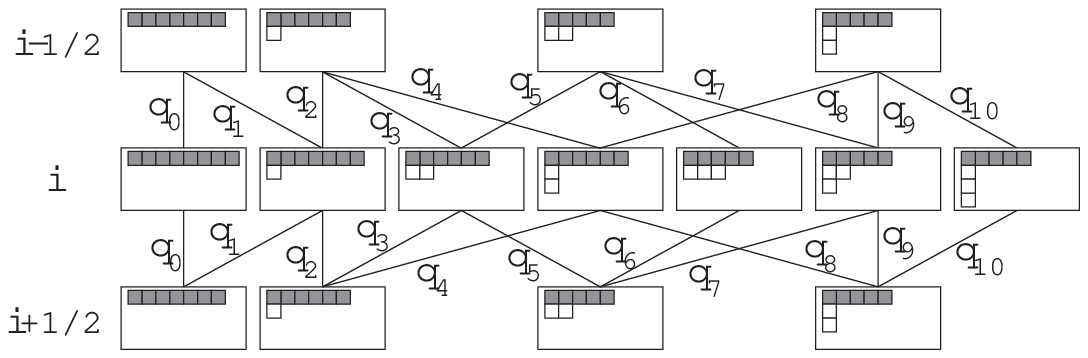}
\begin{example}
Suppose that tableau $\{q_r\}$ go through
paths in the picture illustrated in Figure~\ref{fig:repF4}.
Then we have
\[
\rho(f_i)(v_0\ v_1)
= (v_0\ v_1)
\begin{pmatrix}
	h(\widehat{\emptyset})/h(\widetilde{\emptyset})
	&h(\widehat{\emptyset})/h(\widetilde{\abox})\\
	h(\widehat{\emptyset})/h(\widetilde{\emptyset})
	&h(\widehat{\emptyset})/h(\widetilde{\abox})
\end{pmatrix}
= (v_0\ v_1)
\begin{pmatrix}
	\frac{1}{Q} &\frac{Q-1}{Q}\\
	\frac{1}{Q} &\frac{Q-1}{Q}
\end{pmatrix}
\]
and
\begin{eqnarray*}
\rho(f_i)(v_2\ v_{3}\ v_{4})
&=& (v_2\ v_{3}\ v_{4})
\begin{pmatrix}
	h(\widehat{\abox})/h(\widetilde{\abox})
	&h(\widehat{\abox})/h(\widetilde{\horibox})
	&h(\widehat{\abox})/h(\widetilde{\vertbox})\\
	h(\widehat{\abox})/h(\widetilde{\abox})
	&h(\widehat{\abox})/h(\widetilde{\horibox})
	&h(\widehat{\abox})/h(\widetilde{\vertbox})\\
	h(\widehat{\abox})/h(\widetilde{\abox})
	&h(\widehat{\abox})/h(\widetilde{\horibox})
	&h(\widehat{\abox})/h(\widetilde{\vertbox})
\end{pmatrix}\\
&=& (v_2\ v_{3}\ v_{4})
\begin{pmatrix}
	\frac{Q-1}{Q(Q-2)} &\frac{Q-3}{2(Q-2)} &\frac{Q-1}{2Q}\\
	\frac{Q-1}{Q(Q-2)} &\frac{Q-3}{2(Q-2)} &\frac{Q-1}{2Q}\\
	\frac{Q-1}{Q(Q-2)} &\frac{Q-3}{2(Q-2)} &\frac{Q-1}{2Q}
\end{pmatrix}.
\end{eqnarray*}
Here $v_i$ is the standard vector which corresponds to $q_i$.
Similarly, for the bases
$\langle v_5, v_6, v_7\rangle$ and
$\langle v_8, v_{9}, v_{10}\rangle$
we have the following matrices respectively:
\[
\begin{pmatrix}
	\frac{Q-3}{(Q-1)(Q-4)} &\frac{Q-5}{3(Q-4)} &\frac{2(Q-2)}{3(Q-1)}\\
	\frac{Q-3}{(Q-1)(Q-4)} &\frac{Q-5}{3(Q-4)} &\frac{2(Q-2)}{3(Q-1)}\\
	\frac{Q-3}{(Q-1)(Q-4)} &\frac{Q-5}{3(Q-4)} &\frac{2(Q-2)}{3(Q-1)}
\end{pmatrix},\quad
\begin{pmatrix}
	\frac{Q-1}{Q(Q-3)} &\frac{2(Q-4)}{3(Q-3)} &\frac{Q-1}{3Q}\\
	\frac{Q-1}{Q(Q-3)} &\frac{2(Q-4)}{3(Q-3)} &\frac{Q-1}{3Q}\\
	\frac{Q-1}{Q(Q-3)} &\frac{2(Q-4)}{3(Q-3)} &\frac{Q-1}{3Q}
\end{pmatrix}.
\]
\end{example}

\subsection{Definition of $\rho_{\bfalpha}(s_i)$}
Finally, we define linear maps for $s_i$.
Unfortunately,
we do not have uniform description for $\rho_{\bfalpha}(s_i)$,
except for ``non-reductive'' paths.
So first we define $\rho_{\bfalpha}(s_i)$
for the non-reductive paths.
Then we define $\rho_{\bfalpha}(s_1)$
and $\rho_{\bfalpha}(s_2)$ for ``reductive'' paths one by one.

\subsubsection*{Non-Reductive Case}
In the following, we use notation $\mu\ssubset\lambda$
if a Young diagram $\lambda$
is obtained from a Young diagram $\mu$ by
adding one box.

For $1\leq j\leq i$,
let ${\nu}$, ${\mu}$, ${\lambda}$
be Young diagrams of size $j-1$, $j$ and $j+1$ respectively
such that $\nu\ssubset\mu\ssubset\lambda$.
If a tableau $P$ of $\mathbb{T}(\bfalpha)$
goes through $\widetilde{\nu}$, $\widetilde{\mu}$ and $\widetilde{\lambda}$
at the $(i-2)$-nd, the $(i-1)$-st and the $i$-th coordinate,
then $P$ goes through $\widehat{\nu}$ and $\widehat{\mu}$
at the $(i-3/2)$-th and the $(i-1/2)$-th coordinate.
We call such a tableau {\em non-reductive} at $i$.
If a tableau $P$ does not satisfy the property above,
then we call $P$ {\em reductive} at $i$.

Recall that if $\nu\ssubset\mu\ssubset\lambda$,
then we can define the {\em axial distance} $d = d(\nu, \mu, \lambda)$.
Namely, if $\mu$ differs from $\nu$
in the $r_0$-th row and the $c_0$-th column only,
and $\lambda$ differs from $\mu$
in the $r_1$-th row and the $c_1$-th column only,
then $d = d(\nu, \mu, \lambda)$ is defined by
\begin{equation*}
d = d(\nu, \mu, \lambda) = (c_1 - r_1) - (c_0 - r_0)
 = \left\{
	\begin{array}{ll}
		h_{\lambda}(r_1, c_0) -1 & \mbox{ if } r_0 \geq r_1,\\
		1 - h_{\lambda}(r_0, c_1) & \mbox{ if } r_0 < r_1.
	\end{array}
	\right.
\end{equation*}
Here $h_{\lambda}(i,j)$ is the {\em hook-length} at $(i,j)$ in $\lambda$.

If $|d|\geq 2$,
then there is a unique Young diagram
$\mu^{\prime}\neq \mu$
which satisfies
$\nu\ssubset\mu^{\prime}\ssubset\lambda$.
Let $P'$
be a tableau of shape $\bfalpha$ which are obtained from $P$
by replacing $(i-1)$-st and $(i-1/2)$-th coordinates
of $P$ with $\widetilde{\mu'}$ and $\widehat{\mu'}$ respectively.
For the standard vectors $v_P$ and $v_{P'}$ which correspond to $P$ and $P'$,
we define the linear map $\rho_{\bfalpha}(s_i)$ as follows:
\begin{equation}\label{eq:non-red}
	\rho_{\bfalpha}(s_i)\ :\ (v_{P}, v_{P'})
	\longmapsto (v_{P}, v_{P'})
	\left(
		\begin{array}{cc}
			1/d	&	\left(1-1/d^2\right)/c\\
			c	&	-1/d
		\end{array}
	\right),
\end{equation}
where we can arbitrarily chose $c\in K_0\setminus\{0\}$.
If we put
\begin{equation}\label{eq:ad}
	a_d = 1/d\quad \mbox{and}\quad b_d = 1-a_d^2,
\end{equation}
then the matrix in the expression~\eqref{eq:non-red} is written as follows:
\begin{equation*}
	\left(
		\begin{array}{rr}
			a_{d} & b_{d}/c\\
			c & -a_{d}
		\end{array}
	\right).
\end{equation*}
If $|d_1|=1$,
then there does not exist a
distinct Young diagram
$\mu^{\prime}$
which satisfies
$\nu\ssupset\mu^{\prime}\ssupset\lambda$
other than $\mu$.
In this case, we define $\rho_{\bfalpha}(s_i)$ to be
\[
	\rho_{\bfalpha}(s_i)\ :\ v_{P}
	\longmapsto a_d v_{P}.
\]
Here $a_d$ is the one defined by \eqref{eq:ad}.

\begin{example}
Suppose that a tableau $p_1$ of $\mathbb{T}(\bfalpha)$ goes through
$\widetilde{\emptyset}$, $\widetilde{\abox}$ and $\widetilde{\horibox}$
at the 0-th, the 1-st and the 2-nd coordinates respectively,
then for the standard vector $u_1$ which corresponds to $p_1$ we have
\[
	\rho_{\bfalpha}(s_1) u_1 = u_1.
\]
For the standard vector $v_2$ which corresponds to $p_2$,
a tableau of $\mathbb{T}(\bfalpha)$ which goes  through
$\widetilde{\emptyset}$, $\widetilde{\abox}$ and $\widetilde{\vertbox}$
at the 0-th, the 1-st and the 2-nd coordinates respectively,
we have
\[
	\rho_{\bfalpha}(s_1) u_2 = -u_2.
\]
\end{example}

\begin{example}
Let $\lambda^{(1)} = (3)$, $\lambda^{(2)} = (2,1)$
and $\lambda^{(3)} = (1,1,1)$ be partitions of 3.
Suppose that
tableaux $q_1$ and $q_2$ of $\mathbb{T}(\bfalpha)$ both go through
$\widetilde{\abox}$ and $\widetilde{\horibox}$
at the 1-st and the 2-nd coordinates respectively,
and
tableaux $q_3$ and $q_4$ of $\mathbb{T}(\bfalpha)$ both go through
$\widetilde{\abox}$ and $\widetilde{\vertbox}$
at the 1-st and the 2-nd coordinates respectively.
Further, the tableaux $q_1$, $q_2$, $q_3$ and $q_4$
go through $\widetilde{\lambda^{(1)}}$, $\widetilde{\lambda^{(2)}}$,
$\widetilde{\lambda^{(2)}}$ and $\widetilde{\lambda^{(3)}}$
at the 3-rd coordinates respectively.
Then we have 
\[
	\rho_{\bfalpha}(s_2) (v_1\ v_2\ v_3\ v_4)
= (v_1\ v_2\ v_3\ v_4)
\begin{pmatrix}
1 & 0 & 0 & 0\\
0 & -1/2& 3/(4c) & 0\\
0 & c & 1/2 & 0\\
0 & 0 & 0 & -1
\end{pmatrix}.
\]
Here $v_i$ is the standard vector which corresponds to $q_i$.
\end{example}

\subsubsection*{Reductive Case}
Consider the case a tableau $P$ is reductive at $i$.
So far,
we do not have uniform description for $\rho_{\bfalpha}(s_i)$.
So we define $\rho_{\bfalpha}(s_1)$
and $\rho_{\bfalpha}(s_2)$ one by one.

First we define $\rho_{\bfalpha}(s_1)$.
For tableaux $p_1$ and $p_2$ of $\mathbb{T}(\bfalpha)$
which go through
$(\widetilde{\emptyset}, \widehat{\emptyset},
\widetilde{\emptyset}, \widehat{\emptyset}, \widetilde{\emptyset})$
and
$(\widetilde{\emptyset}, \widehat{\emptyset},
\widetilde{\abox}, \widehat{\emptyset}, \widetilde{\emptyset})$
at the 0-th, the $1-\frac{1}{2}$-th, the 1-st,
the $2-\frac{1}{2}$-th and the 2-nd coordinate
respectively,
let $u_1$ and $u_2$ be the corresponding standard vectors.
Then we define $\rho_{\bfalpha}(s_1)(u_1\ u_2)$ by
\[
\rho_{\bfalpha}(s_1)(u_1\ u_2)
= (u_1\ u_2)
\begin{pmatrix}
	1 & 0\\
	0 & 1
\end{pmatrix}.
\]

For tableaux $p_3$, $p_4$ and $p_5$ of $\mathbb{T}(\bfalpha)$
which go through
$(\widetilde{\emptyset}, \widehat{\emptyset},
\widetilde{\emptyset}, \widehat{\emptyset}, \widetilde{\abox})$,
$(\widetilde{\emptyset}, \widehat{\emptyset},
\widetilde{\abox}, \widehat{\emptyset}, \widetilde{\abox})$
and
$(\widetilde{\emptyset}, \widehat{\emptyset},
\widetilde{\abox}, \widehat{\abox}, \widetilde{\abox})$
at the 0-th, the $1-\frac{1}{2}$-th, the 1-st,
the $2-\frac{1}{2}$-th and the 2-nd coordinate
respectively,
let $u_3$, $u_4$ and $u_5$ be the corresponding standard vectors.
Then we define $\rho_{\bfalpha}(s_1)(u_1\ u_2\ u_3)$ by
\[
\rho_{\bfalpha}(s_1)(u_1\ u_2\ u_3)
= (u_1\ u_2\ u_3)
\begin{pmatrix}
	0 & 1 & 1\\
	\frac{1}{Q-1} & \frac{Q-2}{Q-1} & \frac{-1}{Q-1}\\
	\frac{Q-2}{Q-1} & -\frac{Q-2}{Q-1} & \frac{1}{Q-1}
\end{pmatrix}.
\]

Next we define $\rho_{\bfalpha}(s_2)$.
In the following, we write
\[
p = (\lambda^{(1)}, \lambda^{(2)}, \lambda^{(3)},
\lambda^{(4)}, \lambda^{(5)})
\]
to mean the tableau $p$
goes through $\lambda^{(1)}$, $\lambda^{(2)}$, $\lambda^{(3)}$,
$\lambda^{(4)}$, $\lambda^{(5)}$
at the 1-st, the $(2-\frac{1}{2})$-th,
the 2-nd, the $(3-\frac{1}{2})$-th
and the 3-rd coordinates respectively.

Suppose that
\[
\begin{array}{lr}
\begin{array}{rcl}
q_1 &=& (\widetilde{\emptyset}, \widehat{\emptyset},
\widetilde{\emptyset}, \widehat{\emptyset}, \widetilde{\emptyset}),\\
q_2 &=& (\widetilde{\abox}, \widehat{\emptyset},
\widetilde{\emptyset}, \widehat{\emptyset}, \widetilde{\emptyset}),\\
q_3 &=& (\widetilde{\emptyset}, \widehat{\emptyset},
\widetilde{\abox}, \widehat{\emptyset}, \widetilde{\emptyset}),
\end{array}
&
\begin{array}{rcl}
q_4 &=& (\widetilde{\abox}, \widehat{\emptyset},
\widetilde{\abox}, \widehat{\emptyset}, \widetilde{\emptyset}),\\
q_5 &=& (\widetilde{\abox}, \widehat{\abox},
\widetilde{\abox}, \widehat{\emptyset}, \widetilde{\emptyset}).
\end{array}
\end{array}
\]
Then for the standard vectors $(v_j)_{j=1}^5$
which correspond to $(q_j)_{j=1}^5$
we define $\rho_{\bfalpha}(s_2)(v_1\ v_2\ v_3\ v_4\ v_5)$ by
\[
\rho_{\bfalpha}(s_2)(v_1\ v_2\ v_3\ v_4\ v_5)
= (v_1\ v_2\ v_3\ v_4\ v_5)
\begin{pmatrix}
	1 & 0 & 0& 0 &0\\
	0 & 0 & 0& 1 &1\\
	0 & 0 & 1& 0 &0\\
	0 & \frac{1}{Q-1}   & 0& \frac{Q-2}{Q-1} &\frac{-1}{Q-1}\\
	0 & \frac{Q-2}{Q-1} & 0& -\frac{Q-2}{Q-1} &\frac{1}{Q-1}
\end{pmatrix}.
\]

Assume that
\[
\begin{array}{lr}
\begin{array}{rcl}
q_6 &=& (\widetilde{\emptyset}, \widehat{\emptyset},
\widetilde{\emptyset}, \widehat{\emptyset}, \widetilde{\abox}),\\
q_7 &=& (\widetilde{\abox}, \widehat{\emptyset},
\widetilde{\emptyset}, \widehat{\emptyset}, \widetilde{\abox}),\\
q_8 &=& (\widetilde{\emptyset}, \widehat{\emptyset},
\widetilde{\abox}, \widehat{\emptyset}, \widetilde{\abox}),\\
q_9 &=& (\widetilde{\abox}, \widehat{\emptyset},
\widetilde{\abox}, \widehat{\emptyset}, \widetilde{\abox}),\\
q_{10} &=& (\widetilde{\abox}, \widehat{\abox},
\widetilde{\abox}, \widehat{\emptyset}, \widetilde{\abox}),
\end{array}
&
\begin{array}{rcl}
q_{11} &=& (\widetilde{\emptyset}, \widehat{\emptyset},
\widetilde{\abox}, \widehat{\abox}, \widetilde{\abox}),\\
q_{12} &=& (\widetilde{\abox}, \widehat{\emptyset},
\widetilde{\abox}, \widehat{\abox}, \widetilde{\abox}),\\
q_{13} &=& (\widetilde{\abox}, \widehat{\abox},
\widetilde{\abox}, \widehat{\abox}, \widetilde{\abox}),\\
q_{14} &=& (\widetilde{\abox}, \widehat{\abox},
\widetilde{\horibox}, \widehat{\abox}, \widetilde{\abox}),\\
q_{15} &=& (\widetilde{\abox}, \widehat{\abox},
\widetilde{\vertbox}, \widehat{\abox}, \widetilde{\abox}).
\end{array}
\end{array}
\]
Then for the standard vectors $(v_j)_{j=6}^{15}$
which correspond to $(q_j)_{j=6}^{15}$
we define $\rho_{\bfalpha}(s_2)(v_6\ v_8\ v_{11})$ and
$\rho_{\bfalpha}(s_2)(v_7\ v_9\ v_{10} \ v_{12}\ v_{13}\ v_{14}\ v_{15})$
by
\[
\rho_{\bfalpha}(s_2)(v_6\ v_8\ v_{11})
= (v_6\ v_8\ v_{11})
\begin{pmatrix}
0&1&1\\
\noalign{\medskip}
\frac{1}{(Q-1)}
&\frac{Q-2}{Q-1}
&\frac{-1}{(Q-1)}\\
\noalign{\medskip}
\frac{Q-2}{Q-1}
&-\frac{Q-2}{Q-1}
&\frac{1}{Q-1}
\end{pmatrix}
\]
and
\[
\rho_{\bfalpha}(s_2)(v_7\ v_9\ v_{10}
\ v_{12}\ v_{13}\ v_{14}\ v_{15})
= (v_7\ v_9\ v_{10}
\ v_{12}\ v_{13}\ v_{14}\ v_{15})M_i
\]
Here the matrix $M_{i}$ is
\[
\begin{pmatrix}
\noalign{\medskip}
\frac{1}{Q-1}
&\frac{Q-2}{Q-1}
&\frac{-1}{Q-1}
&\frac{-1}{Q-1}
&\frac{1}{(Q-1)(Q-2)}
&\frac{(Q-1)(Q-2)-2}{2(Q-1)(Q-2)}
&-1/2\\
\noalign{\medskip}
\frac{Q-2}{(Q-1)^{2}}
&\frac{Q^2-3Q+3}{(Q-1)^2}
&\frac{1}{(Q-1)^2}
&\frac{1}{(Q-1)^2}
&\frac{-1}{(Q-1)^2(Q-2)}
&\frac{-Q(Q-3)}{2(Q-1)^2(Q-2)}
&\frac{1}{2(Q-1)}\\
\noalign{\medskip}
\frac{-(Q-2)}{(Q-1)^2}
&\frac{Q-2}{(Q-1)^2}
&\frac{-1}{(Q-1)^2}
&\frac{Q(Q-2)}{(Q-1)^2}
&\frac{1}{(Q-1)^2(Q-2)}
&\frac{Q(Q-3)}{2(Q-1)^2(Q-2)}
&\frac{-1}{2(Q-1)}\\
\noalign{\medskip}
\frac{-(Q-2)}{(Q-1)^2}
&\frac{Q-2}{(Q-1)^2}
&\frac{Q(Q-2)}{(Q-1)^2}
&\frac{-1}{(Q-1)^2}
&\frac{1}{(Q-1)^2(Q-2)}
&\frac {Q(Q-3)}{2(Q-1)^2(Q-2)}
&\frac{-1}{2(Q-1)}\\
\noalign{\medskip}
\frac{Q-2}{(Q-1)^2}
&\frac{-(Q-2)}{(Q-1)^2}
&\frac{1}{(Q-1)^2}
&\frac{1}{(Q-1)^2}
&\frac{(Q-1)^2(Q-2)-1}{(Q-1)^2(Q-2)}
&\frac{-Q(Q-3)}{2(Q-1)^2(Q-2)}
&\frac{1}{2(Q-1)}\\
\noalign{\medskip}
\frac{Q-2}{Q-1}
&\frac{-(Q-2)}{Q-1}
&\frac{1}{Q-1}
&\frac{1}{Q-1}
&\frac{-1}{(Q-1)(Q-2)}
&\frac{Q^2-3Q+4}{2(Q-1)(Q-2)}
&1/2\\
\noalign{\medskip}
\frac{-(Q-2)}{Q-1}
&\frac{Q-2}{Q-1}
&\frac{-1}{Q-1}
&\frac{-1}{Q-1}
&{\frac{1}{(Q-1)(Q-2)}}
&\frac{Q(Q-3)}{2(Q-1)(Q-2)}
&1/2
\end{pmatrix}.
\]
Next assume that
\[
\begin{array}{lr}
\begin{array}{rcl}
q_{16} &=& (\widetilde{\emptyset}, \widehat{\emptyset},
\widetilde{\abox}, \widehat{\abox}, \widetilde{\horibox}),\\
q_{17} &=& (\widetilde{\abox}, \widehat{\emptyset},
\widetilde{\abox}, \widehat{\abox}, \widetilde{\horibox}),\\
q_{18} &=& (\widetilde{\abox}, \widehat{\abox},
\widetilde{\abox}, \widehat{\abox}, \widetilde{\horibox}),
\end{array}
&
\begin{array}{rcl}
q_{19} &=& (\widetilde{\abox}, \widehat{\abox},
\widetilde{\horibox}, \widehat{\abox}, \widetilde{\horibox}),\\
q_{20} &=& (\widetilde{\abox}, \widehat{\abox},
\widetilde{\vertbox}, \widehat{\abox}, \widetilde{\horibox}),\\
q_{21} &=& (\widetilde{\abox}, \widehat{\abox},
\widetilde{\horibox}, \widehat{\horibox}, \widetilde{\horibox}).
\end{array}
\end{array}
\]
Then for the standard vectors $(v_j)_{j=16}^{21}$
which correspond to $(q_j)_{j=16}^{21}$
we define $\rho_{\bfalpha}(s_2)(v_{16}\ v_{17}\ v_{18}\ v_{19}\ v_{20}\ v_{21}
)$ by
\[
\rho_{\bfalpha}(s_2)(v_{16}\ v_{17}\ v_{18}\ v_{19}\ v_{20}\ v_{21})
= (v_{16}\ v_{17}\ v_{18}\ v_{19}\ v_{20}\ v_{21})M_i.
\]
Here the matrix $M_i$ is
\begin{eqnarray*}
\begin{pmatrix}
1
&0
&0
&0
&0
&0\\
\noalign{\medskip}0
&\frac{-1}{(Q-1)}
&\frac{1}{(Q-1)(Q-2)}
&\frac{Q(Q-3)}{2(Q-1)(Q-2)}
&-1/2
&0\\
\noalign{\medskip}0
&\frac{1}{(Q-1)}
&\frac{-1}{(Q-1)(Q-2)}
&\frac{{Q}^{2}-3\,Q+4}{2(Q-1)(Q-2)}
&1/2
&1\\
\noalign{\medskip}0
&1
&\frac{{Q}^{2}-3\,Q+4}{Q(Q-3)(Q-2)}
&\frac{(Q-1)( Q-4)}{2( Q-2)(Q-3)}
&\frac{(Q-1)(Q-4)}{2Q(Q-3)}
&\frac{-1}{(Q-3)}\\
\noalign{\medskip}0
&-1
&\frac{1}{(Q-2)}
&\frac {Q-4}{2(Q-2)}
&1/2
&\frac{-1}{(Q-1)}\\
\noalign{\medskip}0
&0
&\frac{(Q-1)^{2}(Q-4)}{Q(Q-3)(Q-2)}
&\frac{-(Q-1)(Q-4)}{2(Q-2)(Q-3)}
&\frac{-(Q-1)(Q-4)}{2Q(Q-3)}
&\frac{1}{(Q-3)}
\end{pmatrix}.
\end{eqnarray*}

Finally assume that
\[
\begin{array}{lr}
\begin{array}{rcl}
q_{22} &=& (\widetilde{\emptyset}, \widehat{\emptyset},
\widetilde{\abox}, \widehat{\abox}, \widetilde{\vertbox}),\\
q_{23} &=& (\widetilde{\abox}, \widehat{\emptyset},
\widetilde{\abox}, \widehat{\abox}, \widetilde{\vertbox}),\\
q_{24} &=& (\widetilde{\abox}, \widehat{\abox},
\widetilde{\abox}, \widehat{\abox}, \widetilde{\vertbox}),
\end{array}
&
\begin{array}{rcl}
q_{25} &=& (\widetilde{\abox}, \widehat{\abox},
\widetilde{\horibox}, \widehat{\abox}, \widetilde{\vertbox}),\\
q_{26} &=& (\widetilde{\abox}, \widehat{\abox},
\widetilde{\vertbox}, \widehat{\abox}, \widetilde{\vertbox}),\\
q_{27} &=& (\widetilde{\abox}, \widehat{\abox},
\widetilde{\horibox}, \widehat{\vertbox}, \widetilde{\vertbox}).
\end{array}
\end{array}
\]
Then for the standard vectors $(v_j)_{j=22}^{27}$
which correspond to $(q_j)_{j=22}^{27}$
we define $\rho_{\bfalpha}(s_2)(v_{22}\ v_{23}\ v_{24}\ v_{25}\ v_{26}\ v_{27}
)$ by
\[
\rho_{\bfalpha}(s_2)(v_{22}\ v_{23}\ v_{24}\ v_{25}\ v_{26}\ v_{27})
= (v_{22}\ v_{23}\ v_{24}\ v_{25}\ v_{26}\ v_{27})M_i.
\]
Here the matrix $M_i$ is
\begin{eqnarray*}
\begin{pmatrix}
-1
&0
&0
&0
&0
&0\\
\noalign{\medskip}
0
&\frac{1}{(Q-1)}
&\frac{-1}{(Q-1)(Q-2)}
&\frac{-Q(Q-3)}{2(Q-1)( Q-2 )}
&1/2
&0\\
\noalign{\medskip}
0
&\frac{-1}{(Q-1)}
&\frac{1}{(Q-1)(Q-2)}
&\frac{Q(Q-3)}{2(Q-1)(Q-2)}
&1/2
&1\\
\noalign{\medskip}
0
&-1
&\frac{1}{(Q-2)}
&\frac{Q-4}{2(Q-2)}
&1/2
&\frac{-1}{(Q-3)}
\\
\noalign{\medskip}
0
&1
&\frac{1}{(Q-2)}
&\frac{Q(Q-3)}{2(Q-1)(Q-2)}
&\frac{Q-3}{2(Q-1)}
&\frac{-1}{(Q-1)}\\
\noalign{\medskip}
0
&0
&\frac{Q-3}{Q-2}
&\frac{-Q(Q-3)}{2(Q-1)(Q-2)}
&\frac{-(Q-3)}{2(Q-1)}
&\frac{1}{(Q-1)}
\end{pmatrix}.
\end{eqnarray*}

\section{Discussion}\label{sec:dec}

In the previous section, we gave linear maps
$\rho_{\bfalpha}(e_i)$ and
$\rho_{\bfalpha}(f_i)$
for all the tableaux on $\Gamma_n$.
and defined
$\rho_{\bfalpha}(s_i)$ for non-reductive tableaux on $\Gamma_n$.
We also defined $\rho_{\bfalpha}(s_1)$
and $\rho_{\bfalpha}(s_2)$ for the reductive tableaux on $\Gamma_n$. 
(So far, we have further obtained $\rho_{\bfalpha}(s_3)$
for almost all reductive tableaux on $\Gamma_4$.)
These linear maps preserve the relations
in Theorem~1.2 and Theorem~5.3.
Hence they give representations of $A_n(Q)$ for
all $\bfalpha\in\bfLambda_n$
($n = 2-\frac{1}{2}, 2, 3-\frac{1}{2}, 3, 4-\frac{1}{2}$)
and for almost all $\bfalpha\in\bfLambda_4$.

These representations also coincide with the ones calculated
through the Murphy's operators which are introduced
in the paper~\cite{HR} and programmed by Naruse.
Moreover, the traces of the representation matrices above
coincide with the ``characters'' which is defined
by Naruse in the paper~\cite{Na1}.
This means that
the representations we have presented in this note
will be irreducible and define Young's seminormal form
representations of the partition algebras $A_n(Q)$.


\vspace*{2.0cm}
\noindent
Department of Mathematical Sciences\\
Faculty of Science\\
University of the Ryukyus\\
Nishihara-cho, Okinawa 903-0213\\
JAPAN\\
\\
kosuda@math.u-ryukyu.ac.jp

\end{document}